\documentclass[11pt,reqno,twoside]{amsart}
\usepackage{amssymb,amsmath,amsthm,soul,color,paralist}
\usepackage{t1enc}
\usepackage[cp1250]{inputenc}
\usepackage{a4,indentfirst,latexsym}
\usepackage{graphics}
\usepackage{mathrsfs}
\usepackage{cite,enumitem,graphicx}
\usepackage[colorlinks=true,urlcolor=blue,
citecolor=red,linkcolor=blue,linktocpage,pdfpagelabels,
bookmarksnumbered,bookmarksopen]{hyperref}
\usepackage[english]{babel}
\usepackage[left=2.50cm, right=2.50cm, top=2.10cm, bottom=2.10cm]{geometry}
\usepackage[metapost]{mfpic}
\usepackage[colorinlistoftodos]{todonotes}
\usepackage[normalem]{ulem}
\usepackage{comment}

\makeatletter
\providecommand\@dotsep{5}
\def\listtodoname{List of Todos}
\def\listoftodos{\@starttoc{tdo}\listtodoname}
\makeatother

\numberwithin{equation}{section}

\allowdisplaybreaks

\newcommand{\h}{\mathcal{H}_{\e}}
\newcommand{\R}{\mathbb{R}}
\newcommand{\2}{2^{*}_{s}}

\newcommand{\C}{\mathbb{C}}
\newcommand{\ri}{\rightarrow}

\newcommand{\N}{\mathcal{N}}
\newcommand{\M}{\mathcal{M}}
\newcommand{\La}{\Lambda}

\DeclareMathOperator{\dive}{div}
\DeclareMathOperator{\supp}{supp}
\DeclareMathOperator{\e}{\varepsilon}

\newtheorem{prop}{Proposition}[section]
\newtheorem{lem}{Lemma}[section]
\newtheorem{thm}{Theorem}[section]

\newtheorem{cor}{Corollary}[section]

\newtheorem{remark}{Remark}[section]


\date{}

\begin{document}
\title[Concentration phenomena for fractional magnetic NLS equations]
{Concentration phenomena for fractional \\
magnetic NLS equations}

\author[V. Ambrosio]{Vincenzo Ambrosio}
\address{Vincenzo Ambrosio\hfill\break\indent 
Dipartimento di Ingegneria Industriale e Scienze Matematiche \hfill\break\indent
Universit\`a Politecnica delle Marche\hfill\break\indent
Via Brecce Bianche, 12\hfill\break\indent
60131 Ancona (Italy)}
\email{v.ambrosio@staff.univpm.it}

\begin{abstract}
We study the multiplicity and concentration of complex valued solutions for a fractional magnetic Schr\"odinger equation involving a scalar continuous  electric potential
satisfying a local condition and a continuous nonlinearity with subcritical growth. The main results are obtained by applying a penalization technique, generalized Nehari manifold method and Ljusternik-Schnirelman theory. We also prove a Kato's inequality for the fractional magnetic Laplacian which we believe to be useful in the study of other fractional magnetic problems.
\end{abstract}

\maketitle

\section{introduction}

\noindent
In this paper we deal with the multiplicity of solutions $u:\R^{N}\ri \C$ to the following fractional magnetic nonlinear Schr\"odinger equation:
\begin{equation}\label{P}
\e^{2s}(-\Delta)_{A/\varepsilon}^{s}u+V(x)u=f(|u|^{2})u \quad \mbox{ in } \mathbb{R}^{N},
\end{equation}
where $\e>0$ is a small parameter, $N\geq 3$, $s\in (0, 1)$ and $A: \R^{N}\ri \R^{N}$ is a H\"older continuous magnetic potential with exponent $\alpha\in(0,1]$.
Along the paper, we assume that $V: \R^{N}\rightarrow \R$ is a continuous electric potential satisfying the following del Pino and Felmer \cite{DF} type assumptions:
\begin{compactenum}[$(V_1)$]
\item there exists $V_{0}>0$ such that $V_{0}=\inf_{x\in \R^{N}} V(x)$;
\item  there exists a bounded open set $\Lambda\subset \R^{N}$ such that
\begin{equation}
V_{0}<\min_{\partial \Lambda} V \quad \mbox{ and } \quad M=\{x\in \Lambda: V(x)=V_{0}\}\neq \emptyset. 
\end{equation}
\end{compactenum} 
Without loss of generality, we may assume that $0\in M$. 
Moreover, we suppose that the nonlinearity $f:\R\rightarrow \R$ is a  continuous function fulfilling the following conditions:
\begin{compactenum}[$(f_1)$]
\item $f(t)=0$ for $t\leq 0$; 
\item there exists $q\in (2, 2^{*}_{s})$, where $2^{*}_{s}= \frac{2N}{N-2s}$ is the fractional critical exponent, such that 
$$
\lim_{t\rightarrow \infty} \frac{f(t)}{t^{\frac{q-2}{2}}}=0;
$$
\item there exists $\theta\in (2, \2)$ such that $0<\frac{\theta}{2} F(t)\leq t f(t)$ for any $t>0$, where $F(t)=\int_{0}^{t} f(\tau)d\tau$;
\item the function $t\mapsto f(t)$ is increasing in $(0, \infty)$.
\end{compactenum} 
Here $(-\Delta)^{s}_{A}$ is the fractional magnetic Schr\"odinger operator which is defined for any $u\in C^{\infty}_{c}(\R^{N}, \C)$ as
\begin{equation*}
(-\Delta)^{s}_{A}u(x)=C(N,s) \lim_{r\rightarrow 0} \int_{B_{r}^{c}(x)} \frac{u(x)-e^{\imath (x-y)\cdot A(\frac{x+y}{2})} u(y)}{|x-y|^{N+2s}} \,dy,
\quad
C(N,s)=s\frac{4^{s}\Gamma\left(\frac{N+2s}{2}\right)}{\pi^{N/2}\Gamma(1-s)}.
\end{equation*}
This operator has been recently introduced in \cite{DS, I10} and relies essentially on the L\'evy-Khintchine formula for the generator of a general L\'evy process. We would like to observe that, when $s=\frac{1}{2}$, 
the above operator takes inspiration from the definition of a quantized operator corresponding to the classical relativistic Hamiltonian symbol for a relativistic particle of mass $m\geq 0$, that is
$$
\sqrt{(\xi-A(x))^{2}+m^{2}}+V(x), \quad (\xi, x)\in \mathbb{R}^{N}\times \mathbb{R}^{N},
$$
which is the sum of the kinetic energy term involving the magnetic vector potential $A(x)$ and the potential energy term given by the electric scalar potential $V(x)$. 
For the sake of completeness, we emphasize that in the literature other three kinds of quantum relativistic Hamiltonians appear depending on how the kinetic energy term $\sqrt{(\xi-A(x))^{2}+m^{2}}$ is quantized.
As explained in \cite{I10}, these three nonlocal operators are in general different from each other, but coincide when the vector potential $A$ is assumed to be linear.
As $s\rightarrow 1$ and $A$ sufficiently smooth, $(-\Delta)^{s}_{A}$ can be also considered (see \cite{SV}) as  the fractional analogue of the magnetic Laplacian
$$
-\Delta_{A} u=\left(\frac{1}{\imath}\nabla-A(x)\right)^{2}u= -\Delta u -\frac{2}{\imath} A(x) \cdot \nabla u + |A(x)|^{2} u -\frac{1}{\imath} u \dive(A(x)),
$$ 
which plays a fundamental role in quantum mechanics in the description of the dynamics of the particle in a non-relativistic setting. Indeed, in the three-dimensional case, the magnetic field $B$ is exactly the curl of $A$, while for higher dimensions $N\geq 4$, $B$ should be thought of as the 2-form given by $B_{ij} =\partial_{j}A_{k}-\partial_{k}A_{j}$; see \cite{AHS, SS} for more physical background.  Motivated by this fact, many authors  \cite{AF, AFF, AS, ChS, CS, EL, K} studied the existence and multiplicity of nontrivial solutions for the following nonlinear magnetic Schr\"odinger equation: 
\begin{equation}\label{MSE}
-\e^{2}\Delta_{A} u+V(x)u=f(x, |u|^{2})u \quad \mbox{ in } \R^{N}.
\end{equation}
Equation \eqref{MSE} arises when we look for standing wave solutions $\psi(x, t)=u(x)e^{-\imath \frac{E}{\e}t}$, with $E\in \R$, for the following time-dependent Schr\"odinger equation in the presence of an external magnetic field:
$$
\imath \e \frac{\partial \psi}{\partial t}=\left(\frac{\e}{\imath}\nabla-A(x)\right)^{2} \psi+U(x)\psi-f(|\psi|^{2})\psi \quad \mbox{ in } (x, t)\in \R^{N}\times \R,
$$
where $U(x)=V(x)+E$. An important class of solutions of \eqref{MSE} are the so-called semi-classical states which concentrate and develop a spike shape around one, or more, particular points in $\R^{N}$, while vanishing elsewhere as $\e\rightarrow 0$. This interest is related to the fact that the transition from quantum mechanics to classical mechanics can be formally performed by sending $\e\rightarrow 0$.

In the last few years, much attention has been paid to the following fractional magnetic Schr\"odinger equation:
\begin{equation}\label{FMSE}
\e^{2s}(-\Delta)^{s}_{A}u+V(x)u=f(x, |u|^{2})u \quad \mbox{ in } \R^{N}.
\end{equation}
For instance, d'Avenia and Squassina \cite{DS} considered a class of minimization problems in the spirit of results due to Esteban and Lions in \cite{EL}.
In \cite{AD} the author and d'Avenia studied the existence and multiplicity of solutions to \eqref{FMSE} for small $\e>0$ when $f\in C^1$ has a subcritical growth and the potential $V$ satisfies the following global condition due to Rabinowitz \cite{Rab}:
\begin{equation}\label{RVC}
\liminf_{|x|\rightarrow \infty} V(x)>\inf_{x\in \R^{N}} V(x).
\end{equation}
In \cite{Acpde} (see also \cite{ANA2020}) the author used a penalization argument to prove the existence and concentration of nontrivial solutions to \eqref{FMSE} under assumptions $(V_1)$-$(V_2)$ and $(f_1)$-$(f_4)$. 
Further interesting results for nonlocal problems involving the operator $(-\Delta)^{s}_{A}$ can be found in \cite{Ana, Adcds, Adypde, Apisa, FPV, MRZ, NPSV}. 

In the absence of the magnetic vector potential, i.e. $A\equiv 0$, the operator $(-\Delta)^{s}_{A}$ reduces to the celebrated fractional Laplacian operator $(-\Delta)^{s}$ 
which arises in a quite natural way in many different physical contexts in which one has to consider long range anomalous diffusions and transport in highly heterogeneous medium; see \cite{BucurV, DPV} for more details on this topic.
Then \eqref{FMSE} boils down to the following fractional Schr\"odinger equation (see \cite{Laskin})
\begin{equation}\label{FSE}
\e^{2s}(-\Delta)^{s}u+V(x)u=f(x,u) \quad \mbox{ in } \R^{N},
\end{equation} 
for which several existence and multiplicity results under different assumptions on $V$ and $f$ have been established via appropriate variational and topological methods; see \cite{AM, A1, A3, DDPW, DMV, FigS} and references therein.
In particular, Davila et al. \cite{DDPW} proved that if $V\in C^{1, \alpha}(\R^{N})\cap L^{\infty}(\R^{N})$ and $\inf_{x\in \R^{N}} V(x)>0$, then (\ref{FSE}) has multi-peak solutions by using the Lyapunov-Schmidt reduction method. 
Fall et al. \cite{FMV} established necessary and sufficient conditions on the smooth potential $V$ in order to produce concentration of solutions of (\ref{FSE}) as $\e\rightarrow 0$.
Alves and Miyagaki \cite{AM} (see also \cite{A3}) considered the existence and concentration of positive solutions of \eqref{FSE} when $V$ satisfies a local condition and $f$ has subcritical growth at infinity. 
In \cite{A1} the author combined the generalized Nehari manifold approach introduced in \cite{SW} with the Ljusternik-Schnirelman theory to obtain a multiplicity result for \eqref{FSE} under assumptions $(V_1)$-$(V_2)$.

Particularly motivated by \cite{AFF, AM, A1, Acpde, ANA2020, AD}, 
in this paper we deal with the multiplicity and concentration phenomenon  as $\e\rightarrow 0$ of nontrivial solutions $u:\R^N\ri \C$ to \eqref{P}, under assumptions $(V_1)$-$(V_2)$ and $(f_1)$-$(f_4)$.
More precisely, our main result can be stated as follows:
\begin{thm}\label{thm1}
Assume that $(V_1)$-$(V_2)$ and $(f_1)$-$(f_4)$ hold. Then, for any $\delta>0$ such that
$$
M_{\delta}=\{x\in \R^{N}: {\rm dist}(x, M)\leq \delta\}\subset \Lambda,
$$ 
there exists $\e_{\delta}>0$ such that, for any $\e\in (0, \e_{\delta})$, problem \eqref{P} has at least $cat_{M_{\delta}}(M)$ nontrivial solutions. Moreover, if $u_{\e}$ denotes one of these solutions and $x_{\e}$ is a global maximum point of $|u_{\e}|$, then we have 
$$
\lim_{\e\rightarrow 0} V(x_{\e})=V_{0},
$$	
and there exists a constant $C>0$ such that
$$
|u_{\e}(x)|\leq \frac{C\e^{N+2s}}{C\e^{N+2s}+|x-x_{\e}|^{N+2s}} \quad \mbox{ for all } x\in \R^{N}.
$$
\end{thm}
The proof of Theorem \ref{thm1} will be obtained by combining suitable variational techniques and Ljusternik-Schnirelman theory. As in \cite{Acpde}, we adapt the penalization approach in \cite{DF} (see also \cite{AFF}) modifying appropriately the nonlinearity $f$ outside $\Lambda$ and by considering an auxiliary problem. The main feature of the corresponding modified energy functional $J_{\e}$ is that it satisfies all the assumptions of the mountain-pass theorem \cite{AR}. In order to obtain a multiplicity result for the modified problem, 
we use a strategy proposed by Benci and Cerami \cite{BC} which consists in making precise comparisons between the category of some sublevel sets of $J_{\e}$ and the category of the set $M$. 
Since the nonlinearity $f$ is only continuous, the Nehari manifold associated with $J_{\e}$ is not differentiable, so we cannot repeat the same arguments used in \cite{AFF, AD} for $C^{1}$-Nehari manifolds \cite{W}. 
We overcome this obstacle by taking advantage of some abstract critical point theorems from Szulkin and Weth \cite{SW}. We recall that
a similar approach is also used in \cite{A1} where $A\equiv 0$. However, the presence of the magnetic potential creates several difficulties which do not permit to adapt the techniques used in \cite{A1} so that a more accurate analysis will be needed in our situation. Indeed, the regularity assumption on $A$ and the use of the fractional diamagnetic inequality \cite{DS} will play a crucial role to obtain some refined estimates. 
Finally, we need to prove that for $\e>0$ small enough, the solutions $u_{\e}$ of the modified problem are also solutions of the original one. To achieve our purpose, we first prove a Kato's inequality \cite{Kato} for solutions of fractional magnetic problems, which essentially says that if $u:\R^{N}\ri \C$ satisfies $(-\Delta)^{s}_{A}u=f\in L^{1}_{loc}$ in $\R^{N}$ then $|u|: \R^{N}\ri \R$ satisfies $(-\Delta)^{s}|u|\leq \Re({\rm sign}(\bar{u})f)$ in the distributional sense. 
We stress that the proof of this result does not follow the original arguments due to Kato \cite{Kato} in which some auxiliary regularity lemmas are used and a double passage to the limit was done. Due to the nonlocal character of $(-\Delta)^{s}_{A}$ and the presence of the magnetic potential, we choose a suitable test function in the weak formulation of the fractional magnetic problem under consideration and we are able to pass to the limit and to obtain the required inequality. 
In some sense, we generalize the scheme used in \cite{Acpde}. Unfortunately the approach in \cite{Acpde} was based 
on the boundedness of the solution and it has only been used so far on specific examples (see \cite{Ana, Adypde, Adcds}). Thus it is not always clear to distinguish what is the core of the approach and what belongs to the specific problem under study. An important achievement of our paper is the derivation of a general abstract fractional Kato's inequality.
Clearly, with respect to the above mentioned studies, the advantage is the simplicity of the presentation and the "ready to use" aspect of the result. 
In light of this result and applying a comparison argument, we show that $u_{\e}$'s are actually solutions to \eqref{P} as long as $\e>0$ is small enough. 
 We emphasize that Theorem \ref{thm1} completes the study started in \cite{Acpde} because we are now considering the question  related to the multiplicity of \eqref{P}. Moreover, Theorem \ref{thm1} improves and extends in fractional setting Theorem $1.1$ in \cite{AFF} in which only $C^{1}$-nonlinearities were considered.
As far as we know,  this is the first time that penalization method jointly with Ljusternik-Schnirelman theory is used to obtain multiple solutions for \eqref{P} under local conditions on $V$ and the continuity of $f$. 
In view of the arguments used along this paper and in \cite{Acpde}, it is easy to see that a multiplicity result holds even in the critical and supercritical cases considered in \cite{Acpde}, more precisely, when we deal with the following fractional magnetic Schr\"odinger equation with critical growth:
\begin{equation}\label{PCRITICAL}
\e^{2s}(-\Delta)_{A/\varepsilon}^{s}u+V(x)u=f(|u|^{2})u+|u|^{\2-2}u \quad \mbox{ in } \mathbb{R}^{N},
\end{equation}
where $f$ satisfies the following assumptions:
\begin{compactenum}[$(h_1)$]
\item $f(t)=0$ for $t\leq 0$;
\item there exist $C_{0}>0$ and $q, \sigma\in (2, 2^{*}_{s})$ such that 
\begin{align*}
f(t)\geq C_{0} t^{\frac{q-2}{2}} \mbox{ for all } t\geq 0 \, \mbox{ and } \lim_{t\rightarrow \infty} \frac{f(t)}{t^{\frac{\sigma-2}{2}}}=0;
\end{align*}
and $C_{0}>0$ if either $N\geq 4s$, or $2s<N<4s$ and $2^{*}_{s}-2<q<2^{*}_{s}$, $C_{0}>0$ is sufficiently large if $2s<N<4s$ and $2<q\leq 2^{*}_{s}-2$.
\item there exists $\theta\in (2, \sigma)$ such that $0<\frac{\theta}{2} F(t)\leq t f(t)$ for any $t>0$, where $F(t)=\int_{0}^{t} f(\tau)d\tau$;
\item  the function $t\mapsto f(t)$ is increasing in $(0, \infty)$;
\end{compactenum} 
and when we study the following fractional magnetic Schr\"odinger equation with supercritical growth:
\begin{equation}\label{PSUPERCRITICAL}
\e^{2s}(-\Delta)_{A/\varepsilon}^{s}u+V(x)u=|u|^{q-2}u+\lambda |u|^{r-2}u \quad \mbox{ in } \mathbb{R}^{N},
\end{equation}
where $2<q<2^{*}_{s}< r$.
The proofs are only a simple adaptation of the techniques used in this paper with minor modifications. For completeness, we state without proofs the following theorems:
\begin{thm}\label{thm1critical}
Assume that $(V_1)$-$(V_2)$ and $(h_1)$-$(h_4)$ hold. Then, for any $\delta>0$ such that
$$
M_{\delta}=\{x\in \R^{N}: {\rm dist}(x, M)\leq \delta\}\subset \Lambda,
$$ 
there exists $\e_{\delta}>0$ such that, for any $\e\in (0, \e_{\delta})$, problem \eqref{PCRITICAL} has at least $cat_{M_{\delta}}(M)$ nontrivial solutions. Moreover, if $u_{\e}$ denotes one of these solutions and $x_{\e}$ is a global maximum point of $|u_{\e}|$, then we have 
$$
\lim_{\e\rightarrow 0} V(x_{\e})=V_{0}.
$$	
\end{thm}
\begin{thm}\label{thm1supercritical}
Assume that $(V_1)$-$(V_2)$ hold. Then there exists $\lambda_{0}>0$ with the following property: for any $\lambda\in (0, \lambda_{0})$ and for any $\delta>0$ such that
$$
M_{\delta}=\{x\in \R^{N}: {\rm dist}(x, M)\leq \delta\}\subset \Lambda,
$$ 
there exists $\e_{\lambda, \delta}>0$ such that, for any $\e\in (0, \e_{\lambda, \delta})$, problem \eqref{PSUPERCRITICAL} has at least $cat_{M_{\delta}}(M)$ nontrivial solutions. Moreover, if $u_{\e}$ denotes one of these solutions and $x_{\e}$ is a global maximum point of $|u_{\e}|$, then we have 
$$
\lim_{\e\rightarrow 0} V(x_{\e})=V_{0}.
$$	
\end{thm}

The paper is structured as follows. In Section $2$ we introduce the notations and we collect some preliminary results for fractional Sobolev spaces. In Section $3$ we study the modified functional. In Section $4$ we consider the scalar limiting problem. In Section $5$ we provide a multiplicity result for the modified problem. Finally, in Section $6$, we give the proof of Theorem \ref{thm1}.

\section{Preliminaries}\label{sec2}

Fix $s\in (0, 1)$ and we denote by $\mathcal{D}^{s, 2}(\R^{N}, \R)$ the completion of $C^{\infty}_{c}(\R^{N}, \R)$ with respect to the Gagliardo seminorm 
$$
[u]=[u]_{s}=\sqrt{\iint_{\R^{2N}} \frac{|u(x)-u(y)|^{2}}{|x-y|^{N+2s}} dxdy}.
$$
When $N>2s$, we also know (see Theorem $2.2$ in \cite{DPQNA}) that 
$$
\mathcal{D}^{s, 2}(\R^{N}, \R)=\{u\in L^{\2}(\R^{N}, \R): [u]<\infty\}.
$$
We denote by
$H^{s}(\R^{N}, \R)$ the fractional Sobolev space 
$$
H^{s}(\R^{N}, \R)=\{u\in L^{2}(\R^{N}, \R): [u]<\infty\}.
$$
It is well-known that the embedding $H^{s}(\R^{N}, \R)\subset L^{q}(\R^{N}, \R)$ is continuous for all $q\in [2, \2)$ and locally compact for all $q\in [1, \2)$; see \cite{DPV}.\\
Let $L^{2}(\R^{N}, \C)$ be the space of complex-valued functions such that $\|u\|_{L^{2}(\R^{N})}^{2}=\int_{\R^{N}}|u|^{2}\, dx<\infty$ endowed with the inner product 
$\langle u, v\rangle_{L^{2}}=\Re\int_{\R^{N}} u\bar{v}\, dx$, where $\Re(z)$ denotes the real part of $z\in \C$ and $\bar{z}$ is its conjugate. 
Let us denote the magnetic Gagliardo seminorm by 
$$
[u]_{A}=[u]_{s, A}=\sqrt{\iint_{\R^{2N}} \frac{|u(x)-e^{\imath (x-y)\cdot A(\frac{x+y}{2})} u(y)|^{2}}{|x-y|^{3+2s}} \, dxdy},
$$
and consider
$$
\mathcal{D}_{A}^{s, 2}(\R^N,\C)=
\left\{
u\in L^{2_s^*}(\R^N,\C) : [u]^{2}_{A}<\infty
\right\}.
$$
Set $A_{\e}(x)=A(\e x)$ and $V_{\e}(x)=V(\e x)$. Then we define the Hilbert space
$$
\h=
\left\{
u\in \mathcal{D}_{A_{\e}}^{s, 2}(\R^N,\C): \int_{\R^{N}} V_{\e}(x) |u|^{2}\, dx <\infty
\right\}
$$ 
endowed with the scalar product
\begin{align*}
&\langle u , v \rangle_{\e}= \Re\iint_{\R^{2N}} \frac{(u(x)-e^{\imath(x-y)\cdot A_{\e}(\frac{x+y}{2})} u(y))\overline{(v(x)-e^{\imath(x-y)\cdot A_{\e}(\frac{x+y}{2})}v(y))}}{|x-y|^{N+2s}} \,dx dy+\Re \int_{\R^{N}} V_{\e}(x) u \bar{v} \,dx
\end{align*}
for all $u, v\in \h$, and let
$$
\|u\|_{\e}=\sqrt{\langle u , u \rangle_{\e}}.
$$
In what follows we list some useful technical lemmas which will be frequently used along the paper; see \cite{AD, DS} for more details.
\begin{thm}\cite{AD, DS}\label{density}
The space $\h$ is complete and $C_c^\infty(\R^N,\C)$ is dense in $\h$. 
\end{thm}

\begin{lem}\label{DI}\cite{DS}
If $u\in \mathcal{D}^{s,2}_{A}(\R^{N}, \C)$ then $|u|\in \mathcal{D}^{s, 2}(\R^{N}, \R)$ and we have
$$
[|u|]\leq [u]_{A}.
$$
\end{lem}

\begin{thm}\label{Sembedding}\cite{DS}
	The space $\h$ is continuously embedded in $L^{r}(\R^{N}, \C)$ for all $r\in [2, 2^{*}_{s}]$, and compactly embedded in $L^{r}(B_{R}, \C)$ for all $R>0$ and any $r\in [1, 2^{*}_{s})$.\\
\end{thm}

\begin{lem}\label{aux}\cite{AD}
If $u\in H^{s}(\R^{N}, \R)$ and $u$ has compact support, then $w=e^{\imath A(0)\cdot x} u \in \h$.
\end{lem}

\noindent
We also recall a fractional version of Lions lemma.
\begin{lem}\label{Lions}\cite{FQT}
	Let $s\in (0, 1)$ and $R>0$. If $(u_{n})$ is a bounded sequence in $H^{s}(\R^{N}, \R)$ and if 
	\begin{equation*}
	\lim_{n\rightarrow \infty} \sup_{y\in \R^{N}} \int_{B_{R}(y)} |u_{n}|^{2} dx=0,
	\end{equation*}
	then $u_{n}\rightarrow 0$ in $L^{r}(\R^{N}, \R)$ for all $r\in (2, 2^{*}_{s})$.
\end{lem}

\section{variational setting and the modified problem}
Using the change of variable $u(x)\mapsto u(\e x)$, we see that (\ref{P}) is equivalent to 
\begin{equation}\label{R}
(-\Delta)^{s}_{A_{\e}} u + V_{\e}(x)u =  f(|u|^{2})u  \quad\mbox{ in } \mathbb{R}^{N}.
\end{equation}
Fix $k>1$ and $a>0$ such that $f(a)=\frac{V_{0}}{k}$, and we define the function
$$
\tilde{f}(t)=
\begin{cases}
f(t) & \text{ if $t \leq a$,} \\
\frac{V_{0}}{k}    & \text{ if $t >a$}.
\end{cases}
$$ 
We introduce the  penalized nonlinearity $g: \R^{N}\times \R\rightarrow \R$ given by
$$
g(x, t)=\chi_{\Lambda}(x)f(t)+(1-\chi_{\Lambda}(x))\tilde{f}(t),
$$
where $\chi_{\Lambda}$ is the characteristic function of $\Lambda$, and  we set $G(x, t)=\int_{0}^{t} g(x, \tau)\, d\tau$.\\
By $(f_1)$-$(f_4)$, it follows that $g$ is a Carath\'eodory function satisfying the following properties:
\begin{compactenum}[($g_1$)]
\item $\lim_{t\rightarrow 0} g(x, t)=0$ uniformly in $x\in \R^{N}$;
\item $g(x, t)\leq f(t)$ for any $x\in \R^{N}$ and $t>0$;
\item $(i)$ $0< \frac{\theta}{2} G(x, t)\leq g(x, t)t$ for any $x\in \Lambda$ and $t>0$, \\
$(ii)$ $0\leq  G(x, t)\leq g(x, t)t\leq \frac{V_{0}}{k}t$ for any $x\in \Lambda^{c}$ and $t>0$;
\item for any $x\in \Lambda$, the function $t\mapsto g(x,t)$ is increasing in $(0, \infty)$, and for any $x\in \Lambda^{c}$, the function $t\mapsto g(x,t)$ is increasing in $(0, a)$.
\end{compactenum}
We used the notation $A^{c}=\R^{N}\setminus A$ for $A\subset \R^{N}$. Set $g_{\e}(x, t)=g(\e x, t)$ and we consider the following modified problem 
\begin{equation}\label{MPe}
\left\{
\begin{array}{ll}
(-\Delta)^{s}_{A_{\e}} u + V_{\e}(x)u =   g_{\e}(x, |u|^{2})u \quad \mbox{ in } \mathbb{R}^{N}, \\
u\in \h.
\end{array}
\right.
\end{equation}
Let us note that if $u$ is a solution of (\ref{MPe}) such that 
\begin{equation}\label{ue}
|u(x)|\leq \sqrt{a} \quad \mbox{ for all } x\in \Lambda^{c}_{\e},
\end{equation}
where $\Lambda_{\e}=\{x\in \R^{N}: \e x\in \Lambda\}$, then $u$ is also a solution of (\ref{R}).
In order to study weak solutions of \eqref{MPe}, we look for critical points of the functional $J_{\e}: \h\rightarrow \R$ defined as
\begin{align*}
J_{\e}(u)=\frac{1}{2}\|u\|_{\e}^{2}-\frac{1}{2}\int_{\R^{N}} G_{\e} (x, |u|^{2})\, dx.
\end{align*}
It is easy to check that $J_{\e}\in C^{1}(\h, \R)$ and that its differential is given by
\begin{align*}
\langle J_{\e}'(u), v\rangle =\langle u, v\rangle_{\e}-\Re\int_{\R^{N}} g_{\e} (x, |u|^{2})u\bar{v} \,dx \quad \mbox{ for all } u, v\in \h.
\end{align*}
Therefore, weak solutions to (\ref{MPe}) can be found as critical points of $J_{\e}$.
Since we are looking for multiple critical points of the functional $J_{\e}$, we shall consider it constrained to an appropriated subset of $\h$.
More precisely, we introduce the Nehari manifold associated with $J_{\e}$, namely
\begin{equation*}
\mathcal{N}_{\e}= \{u\in \h \setminus \{0\} : \langle J_{\e}'(u), u \rangle =0\}.
\end{equation*}
From the growth conditions of $g$ and Theorem \ref{Sembedding}, we see that 
for any fixed $u\in \mathcal{N}_{\e}$ and $\delta>0$ small enough
\begin{align*}
0=\langle J'_{\e}(u), u\rangle&=\|u\|_{\e}^{2}-\int_{\R^{N}} g_{\e} (x, |u|^{2})|u|^{2}\,dx\\
&\geq \|u\|_{\e}^{2}-\delta C_{1} \|u\|_{\e}^{2}-C_{\delta}\|u\|_{\e}^{\2} \\
&\geq C_{2}\|u\|^{2}_{\e}-C_{3}\|u\|_{\e}^{\2},
\end{align*}
so there exists $r>0$ independent of  $u$ such that 
\begin{equation}\label{uNr}
\|u\|_{\e}\geq r \quad\mbox{ for all } u\in \mathcal{N}_{\e}.
\end{equation}
Let us consider
\begin{equation*}
\mathcal{H}_{\e}^{+}= \{u\in \mathcal{H}_{\e} : |{\rm supp}(|u|) \cap \La_{\e}|>0\}\subset \mathcal{H}_{\e}. 
\end{equation*}
Let $\mathbb{S}_{\e}$ be the unit sphere of $\mathcal{H}_{\e}$ and we denote by $\mathbb{S}_{\e}^{+}= \mathbb{S}_{\e}\cap \mathcal{H}_{\e}^{+}$.
We observe that $\mathcal{H}_{\e}^{+}$ is open in $\mathcal{H}_{\e}$.
By the definition of $\mathbb{S}_{\e}^{+}$ and the fact that $\mathcal{H}_{\e}^{+}$ is open in $\mathcal{H}_{\e}$, it follows that $\mathbb{S}_{\e}^{+}$ is an incomplete $C^{1,1}$-manifold of codimension $1$, modelled on $\mathcal{H}_{\e}$ and contained in the open $\mathcal{H}_{\e}^{+}$; see \cite{FJ, SW}. Then, $\mathcal{H}_{\e}=T_{u} \mathbb{S}_{\e}^{+} \oplus \mathbb{R} u$ for each $u\in \mathbb{S}_{\e}^{+}$, where
\begin{equation*}
T_{u} \mathbb{S}_{\e}^{+}= \{v \in \mathcal{H}_{\e} : \langle u, v\rangle_{\e}=0\}.
\end{equation*} 

Next we prove that $J_{\e}$ possesses a mountain pass geometry \cite{AR}.
\begin{lem}\label{MPG}
\begin{compactenum}[$(i)$]
The functional $J_{\e}$ satisfies the following conditions:
\item $J_{\e}(0)=0$;
\item there exist $\alpha, \rho>0$ such that $J_{\e}(u)\geq \alpha$ for any $u\in \h$ such that $\|u\|_{\e}=\rho$;
\item there exists $e\in \h$ such that $\|e\|_{\e}>\rho$ and $J_{\e}(e)<0$.
\end{compactenum}
\end{lem}
\begin{proof}
Clearly, $J_{\e}(0)=0$.
By $(g_1)$ and $(g_2)$, for all $\delta>0$ there exists $C_{\delta}>0$ such that 
$$
|G_{\e} (x, t^2)|\leq \delta |t|^{2}+C_{\delta}|t|^{\2} \quad \mbox{ for every } (x, t)\in \R^{N}\times \R.
$$
This fact combined with Theorem \ref{Sembedding} implies that for all $u\in \h$
$$
J_{\e}(u)\geq \left(\frac{1}{2}-\delta C\right)\|u\|_{\e}^{2}-C'_{\delta} \|u\|^{\2}_{\e}
$$
which implies that $(ii)$ holds true.
Concerning $(iii)$, for each $u\in \mathcal{H}^{+}_{\e}$ and $t>0$, we get
\begin{align*}
J_{\e}(tu)&\leq \frac{t^{2}}{2} \|u\|^{2}_{\e}-\frac{1}{2}\int_{\Lambda_{\e}} F(t^{2}|u|^{2})\, dx \\
&\leq \frac{t^{2}}{2}\|u\|^{2}_{\e}-Ct^{\theta} \int_{\Lambda_{\e}} |u|^{\theta}\, dx+C'|{\rm supp}(|u|)\cap \Lambda_{\e}|
\end{align*}
where we used $(g_2)$ and $(f_3)$. Since $\theta>2$ we see that $J_{\e}(tu)\rightarrow -\infty$ as $t\rightarrow \infty$. 
\end{proof}

Since $f$ is merely continuous, the next results will be fundamental to overcome the non-differentiability of $\mathcal{N}_{\e}$ and the incompleteness of $\mathbb{S}_{\e}^{+}$.
\begin{lem}\label{lem2.3MAW}
Assume that $(V_1)$-$(V_2)$ and $(f_1)$-$(f_4)$ hold. Then, 
\begin{compactenum}[$(i)$]
\item For each $u\in \mathcal{H}_{\e}^{+}$, let $h:\mathbb{R}_{+}\rightarrow \mathbb{R}$ be defined by $h_{u}(t)= J_{\e}(tu)$. Then, there is a unique $t_{u}>0$ such that 
\begin{align*}
&h'_{u}(t)>0 \mbox{ in } (0, t_{u}),\\
&h'_{u}(t)<0 \mbox{ in } (t_{u}, \infty).
\end{align*}
\item There exists $\tau>0$ independent of $u$ such that $t_{u}\geq \tau$ for any $u\in \mathbb{S}_{\e}^{+}$. Moreover, for each compact set $\mathbb{K}\subset \mathbb{S}_{\e}^{+}$ there is a positive constant $C_{\mathbb{K}}$ such that $t_{u}\leq C_{\mathbb{K}}$ for any $u\in \mathbb{K}$.
\item The map $\hat{m}_{\e}: \mathcal{H}_{\e}^{+}\rightarrow \mathcal{N}_{\e}$ given by $\hat{m}_{\e}(u)= t_{u}u$ is continuous and $m_{\e}= \hat{m}_{\e}|_{\mathbb{S}_{\e}^{+}}$ is a homeomorphism between $\mathbb{S}_{\e}^{+}$ and $\mathcal{N}_{\e}$. Moreover, $m_{\e}^{-1}(u)=\frac{u}{\|u\|_{\e}}$.
\item If there is a sequence $(u_{n})\subset \mathbb{S}_{\e}^{+}$ such that ${\rm dist}(u_{n}, \partial \mathbb{S}_{\e}^{+})\rightarrow 0$, then $\|m_{\e}(u_{n})\|_{\e}\rightarrow \infty$ and $J_{\e}(m_{\e}(u_{n}))\rightarrow \infty$.
\end{compactenum}

\end{lem}

\begin{proof}
$(i)$ We note that $h_{u}\in C^{1}(\mathbb{R}_{+}, \mathbb{R})$, and arguing as in the proof of Lemma \ref{MPG} it is easy to verify that $h_{u}(0)=0$, $h_{u}(t)>0$ for $t>0$ small enough and $h_{u}(t)<0$ for $t>0$ sufficiently  large. Therefore, $\max_{t\geq 0} h_{u}(t)$ is achieved at some $t_{u}>0$ verifying $h'_{u}(t_{u})=0$ and $t_{u}u\in \mathcal{N}_{\e}$. 
Next we claim the uniqueness of such a $t_{u}$. Let $t_{1}>t_{2}>0$ be such that $h'_{u}(t_{i})=0$ for $i=1, 2$, that is $\|u\|^{2}_{\e}=\int_{\R^{N}} g_{\e}(x, |t_{i}u|^{2})|u|^{2}\, dx$ for $i=1, 2$. By using the definition of $g$, $(f_4)$, $(g_4)$ and $u\in \mathcal{H}^{+}_{\e}$, we see that
\begin{align*}
0&=\int_{\R^{N}} \left[g_{\e}(x, |t_{1}u|^{2})|u|^{2}-g_{\e}(x, |t_{2}u|^{2})|u|^{2}\right]\, dx\\
&=\int_{\Lambda^{c}_{\e}} \left[g_{\e}(x, |t_{1}u|^{2})|u|^{2}-g_{\e}(x, |t_{2}u|^{2})|u|^{2}\right]\, dx+\int_{\Lambda_{\e}} \left[f(|t_{1}u|^{2})|u|^{2}-f(|t_{2}u|^{2})|u|^{2}\right]\, dx\\
&>\int_{\Lambda^{c}_{\e}\cap \{ |t_{2}u|^{2}\leq a<|t_{1}u|^{2} \}}\left[\frac{V_{0}}{k}|u|^{2}- f(|t_{2}u|^{2})|u|^{2}\right]\, dx+\int_{\Lambda^{c}_{\e}\cap\{|t_{1}u|^{2}\leq a\}} \left[f(|t_{1}u|^{2})|u|^{2}-f(|t_{2}u|^{2})|u|^{2}\right]\, dx\geq 0
\end{align*}
which gives a contradiction.

\noindent
$(ii)$ Let $u\in \mathbb{S}_{\e}^{+}$. By $(i)$ there exists $t_{u}>0$ such that $h_{u}'(t_{u})=0$, or equivalently
\begin{equation*}
t_{u}= \int_{\mathbb{R}^{N}} g_{\e}(x, |t_{u}u|^{2}) t_{u}|u|^{2} \,dx. 
\end{equation*}
By assumptions $(g_{1})$ and $(g_{2})$, given $\xi>0$ there exists a positive constant $C_{\xi}$ such that
\begin{equation}\label{BERTSCH}
|g_{\e}(x, t)|\leq \xi + C_{\xi} |t|^{\frac{\2-2}{2}}, \quad \mbox{ for every } (x, t)\in \R^{N}\times \mathbb{R}.
\end{equation}
Thus \eqref{BERTSCH} and  Theorem \ref{Sembedding} yield
\begin{align*}
t_{u}\leq \xi t_{u}C_{1} + C_{\xi} C_{2} t^{\2-1}_{u}. 
\end{align*}
Taking $\xi>0$ sufficiently small, we obtain that there exists $\tau>0$, independent of $u$, such that $t_{u}\geq \tau$. Now, let $\mathbb{K}\subset \mathbb{S}_{\e}^{+}$ be a compact set and we show that $t_{u}$ can be estimated from above by a constant depending on $\mathbb{K}$. Assume by contradiction that there exists a sequence $(u_{n})\subset \mathbb{K}$ such that $t_{n}=t_{u_{n}}\rightarrow \infty$. Therefore, there exists $u\in \mathbb{K}$ such that $u_{n}\rightarrow u$ in $\mathcal{H}_{\e}$. From $(iii)$ in Lemma \ref{MPG} we get 
\begin{equation}\label{cancMAW}
J_{\e}(t_{n}u_{n})\rightarrow -\infty. 
\end{equation}
Fix $v\in \mathcal{N}_{\e}$. Then, using the fact that $\langle J_{\e}'(v), v \rangle=0$, and assumption $(g_{3})$, we can infer
\begin{align}\label{cancanMAW}
J_{\e}(v)&= J_{\e}(v)- \frac{1}{\theta} \langle J_{\e}'(v), v \rangle \nonumber\\
&=\left(\frac{1}{2}-\frac{1}{\theta}\right) \|v\|_{\e}^{2}+ \int_{\mathbb{R}^{N}} \frac{1}{\theta} g_{\e}(x, |v|^{2})|v|^{2}- \frac{1}{2} G_{\e}(x, |v|^{2})]\, dx \nonumber\\
&\geq \left(\frac{1}{2}-\frac{1}{\theta}\right) \|v\|_{\e}^{2}+ \frac{1}{\theta}\int_{\Lambda_{\e}^{c}} [g_{\e}(x, |v|^{2})|v|^{2}- \frac{\theta}{2} G_{\e}(x, |v|^{2})]\, dx \nonumber \\
&\geq \left(\frac{\theta-2}{2\theta}\right)\|v\|_{\e}^{2} -\left(\frac{\theta-2}{2\theta}\right) \frac{1}{k}\int_{\Lambda^{c}_{\e}} V(\e x) |v|^{2}\,dx \nonumber\\
& \geq \left(\frac{\theta-2}{2\theta}\right) \left(1- \frac{1}{k}\right) \|v\|_{\e}^{2}.
\end{align}
Taking into account that $(t_{u_{n}}u_{n})\subset \mathcal{N}_{\e}$ and $\|t_{u_{n}}u_{n}\|_{\e}=t_{u_{n}}\ri \infty$, from \eqref{cancanMAW} we deduce that \eqref{cancMAW} does not hold. 

\noindent
$(iii)$ First, we note that $\hat{m}_{\e}$, $m_{\e}$ and $m_{\e}^{-1}$ are well defined. Indeed, by $(i)$, for each $u\in \mathcal{H}_{\e}^{+}$ there exists a unique $m_{\e}(u)\in \mathcal{N}_{\e}$. On the other hand, if $u\in \mathcal{N}_{\e}$ then $u\in \mathcal{H}_{\e}^{+}$. Otherwise, if $u\notin \mathcal{H}_{\e}^{+}$, we have
\begin{equation*}
|{\rm supp} (|u|) \cap \Lambda_{\e}|=0, 
\end{equation*}
which together with $(g_3)$-$(ii)$ gives 
\begin{align}\label{ter1MAW}
0<\|u\|_{\e}^{2}&= \int_{\mathbb{R}^{N}} g_{\e}(x, |u|^{2}) |u|^{2} \, dx \nonumber \\
&= \int_{\La_{\e}} g_{\e}(x, |u|^{2}) |u|^{2} \, dx + \int_{\La^{c}_{\e}} g_{\e}(x, |u|^{2}) |u|^{2} \, dx \nonumber \\
&= \int_{\La^{c}_{\e}} g_{\e}(x, |u|^{2}) |u|^{2} \, dx \nonumber \\
&\leq \frac{1}{k} \int_{\La^{c}_{\e}} V_{\e}(x) |u|^{2} \,dx \leq \frac{1}{k} \|u\|_{\e}^{2}.  
\end{align}
Using  \eqref{ter1MAW} we get
\begin{align*}
0<\|u\|_{\e}^{2} \leq \frac{1}{k} \|u\|_{\e}^{2}
\end{align*}
and this leads to a contradiction because $k>1$. 
Consequently, $m_{\e}^{-1}(u)= \frac{u}{\|u\|_{\e}}\in \mathbb{S}_{\e}^{+}$, $m_{\e}^{-1}$ is well defined and continuous. \\
Let $u\in \mathbb{S}_{\e}^{+}$. Then,
\begin{align*}
m_{\e}^{-1}(m_{\e}(u))= m_{\e}^{-1}(t_{u}u)= \frac{t_{u}u}{\|t_{u}u\|_{\e}}= \frac{u}{\|u\|_{\e}}=u
\end{align*}
from which $m_{\e}$ is a bijection. Now, our aim is to prove that $\hat{m}_{\e}$ is a continuous function. 
Let $(u_{n})\subset \mathcal{H}_{\e}^{+}$ and $u\in \mathcal{H}_{\e}^{+}$ such that $u_{n}\rightarrow u$ in $\mathcal{H}_{\e}^{+}$. 
Hence,
\begin{equation*}
\frac{u_{n}}{\|u_{n}\|_{\e}}\rightarrow \frac{u}{\|u\|_{\e}} \quad \mbox{ in } \mathcal{H}_{\e}.
\end{equation*} 
Set $v_{n}= \frac{u_{n}}{\|u_{n}\|_{\e}}$ and $t_{n}=t_{v_{n}}$. By $(ii)$ there exists $t_{0}>0$ such that $t_{n}\rightarrow t_{0}$. Since $t_{n}v_{n}\in \mathcal{N}_{\e}$ and $\|v_{n}\|_{\e}=1$, we have 
\begin{equation*}
t^{2}_{n}= \int_{\mathbb{R}^{N}} g_{\e}(x, |t_{n}v_{n}|^{2}) |t_{n}v_{n}|^{2} \, dx. 
\end{equation*}
Passing to the limit as $n\rightarrow \infty$ we obtain
\begin{equation*}
t^{2}_{0}= \int_{\mathbb{R}^{N}} g_{\e}(x, |t_{0}v|^{2}) |t_{0} v|^{2}\, dx, 
\end{equation*}
where $v= \frac{u}{\|u\|_{\e}}$, which implies that $t_{0}v\in \mathcal{N}_{\e}$. By $(i)$ we deduce that $t_{v}= t_{0}$, and this shows that
\begin{equation*}
\hat{m}_{\e}(u_{n})= \hat{m}_{\e}\left(\frac{u_{n}}{\|u_{n}\|_{\e}}\right)\rightarrow \hat{m}_{\e}\left(\frac{u}{\|u\|_{\e}}\right)=\hat{m}_{\e}(u) \quad \mbox{ in } \mathcal{H}_{\e}.  
\end{equation*}
Therefore $\hat{m}_{\e}$ and $m_{\e}$ are continuous functions. 

\noindent
$(iv)$ Let $(u_{n})\subset \mathbb{S}_{\e}^{+}$ be such that ${\rm dist}(u_{n}, \partial \mathbb{S}_{\e}^{+})\rightarrow 0$. Observe that, by the Sobolev inequality and $(V_1)$-$(V_2)$, for each $p\in [2, 2^{*}_{s}]$ there exists $\kappa_{p}>0$ such that it holds
\begin{align*}
\|u_{n}\|^{p}_{L^{p}(\La_{\e})} &\leq \inf_{v\in \partial \mathbb{S}_{\e}^{+}} \|u_{n}- v\|^{p}_{L^{p}(\La_{\e})}\\
&\leq \kappa_{p} \inf_{v\in \partial \mathbb{S}_{\e}^{+}} \|u_{n}- v\|^{p}_{\e} \\
&\leq \kappa_{p} \, {\rm dist}(u_{n}, \partial \mathbb{S}_{\e}^{+})^{p}, \quad \mbox{ for all } n\in \mathbb{N}.
\end{align*}
Then, by the above inequality, $(g_{1})$, $(g_{2})$ and $(g_{3})$-$(ii)$, we can infer that, for all $t>0$,
\begin{align*}
\frac{1}{2}\int_{\mathbb{R}^{N}} G_{\e}(x, |tu_{n}|^{2})\, dx &= \frac{1}{2}\int_{\La^{c}_{\e}} G_{\e}(x, |tu_{n}|^{2})\, dx + \frac{1}{2}\int_{\La_{\e}} G_{\e}(x, |tu_{n}|^{2})\, dx\\
&\leq \frac{t^{2}}{2k} \int_{\La^{c}_{\e}} V(\e x) |u_{n}|^{2} \,dx+ \int_{\La_{\e}} \frac{1}{2}F(|tu_{n}|^{2})\, dx\\
&\leq \frac{t^{2}}{2k} \|u_{n}\|_{\e}^{2} + C_{1}t^{2} \int_{\La_{\e}} |u_{n}|^{2} \,dx+ C_{2} t^{\2} \int_{\La_{\e}} |u_{n}|^{\2} \,dx\\
&\leq \frac{t^{2}}{2k}+ C_{1}\kappa_{2} t^{2} {\rm dist}(u_{n}, \partial \mathbb{S}_{\e}^{+})^{2} + C_{2}\kappa_{\2} t^{\2} {\rm dist}(u_{n}, \partial \mathbb{S}_{\e}^{+})^{\2}
\end{align*}
from which
\begin{equation}\label{ter2MAW}
\frac{1}{2}\limsup_{n\rightarrow \infty} \int_{\mathbb{R}^{N}} G_{\e}(x, |tu_{n}|^{2})\, dx \leq \frac{t^{2}}{2k} \quad \mbox{ for all } t>0. 
\end{equation}
Bearing in mind the definition of $m_{\e}(u_{n})$, and by using \eqref{ter2MAW}, we have
\begin{align*}
\liminf_{n\rightarrow \infty} J_{\e}(m_{\e}(u_{n}))&\geq \liminf_{n\rightarrow \infty} J_{\e}(tu_{n})\\
&= \liminf_{n\rightarrow \infty} \left[ \frac{1}{2} \|tu_{n}\|_{\e}^{2} - \frac{1}{2}\int_{\mathbb{R}^{N}} G_{\e}(x, |tu_{n}|^{2})\, dx\right]\\
&\geq \left(\frac{1}{2} - \frac{1}{2k} \right)t^{2} \quad \mbox{ for all } t>0.
\end{align*}
Recalling that $k>1$ we get
\begin{equation*}
\lim_{n\rightarrow \infty} J_{\e}(m_{\e}(u_{n}))= \infty. 
\end{equation*}
Moreover, by the definition of $J_{\e}$, we see that
\begin{equation*}
\liminf_{n\ri \infty} \frac{1}{2}\|m_{\e}(u_{n})\|^{2}_{\e}\geq \liminf_{n\ri \infty} J_{\e}(m_{\e}(u_{n}))=\infty.
\end{equation*}
This completes the proof of the lemma.
\end{proof}

\noindent
Let us consider the maps 
\begin{equation*}
\hat{\psi}_{\e}: \mathcal{H}_{\e}^{+} \rightarrow \mathbb{R} \quad \mbox{ and } \quad \psi_{\e}: \mathbb{S}_{\e}^{+}\rightarrow \mathbb{R}, 
\end{equation*}
by setting $\hat{\psi}_{\e}(u)= J_{\e}(\hat{m}_{\e}(u))$ and $\psi_{\e}=\hat{\psi}_{\e}|_{\mathbb{S}_{\e}^{+}}$. 

The next result is a direct consequence of Lemma \ref{lem2.3MAW} and Corollary $10$ in \cite{SW}. 
\begin{prop}\label{prop2.1MAW}
Assume that $(V_{1})$-$(V_{2})$ and $(f_{1})$-$(f_{4})$ hold. Then, 
\begin{compactenum}[$(a)$]
\item $\hat{\psi}_{\e} \in C^{1}(\mathcal{H}_{\e}^{+}, \mathbb{R})$ and 
\begin{equation*}
\langle \hat{\psi}_{\e}'(u), v\rangle = \frac{\|\hat{m}_{\e}(u)\|_{\e}}{\|u\|_{\e}} \langle J_{\e}'(\hat{m}_{\e}(u)), v\rangle, 
\end{equation*}
for every $u\in \mathcal{H}_{\e}^{+}$ and $v\in \mathcal{H}_{\e}$. 
\item $\psi_{\e} \in C^{1}(\mathbb{S}_{\e}^{+}, \mathbb{R})$ and 
\begin{equation*}
\langle \psi_{\e}'(u), v \rangle = \|m_{\e}(u)\|_{\e} \langle J_{\e}'(m_{\e}(u)), v\rangle, 
\end{equation*}
for every $v\in T_{u}\mathbb{S}_{\e}^{+}$.
\item If $(u_{n})$ is a Palais-Smale sequence for $\psi_{\e}$, then $(m_{\e}(u_{n}))$ is a Palais-Smale sequence for $J_{\e}$. If $(u_{n})\subset \mathcal{N}_{\e}$ is a bounded Palais-Smale sequence for $J_{\e}$, then $(m_{\e}^{-1}(u_{n}))$ is a Palais-Smale sequence for  $\psi_{\e}$. 
\item $u$ is a critical point of $\psi_{\e}$ if and only if $m_{\e}(u)$ is a nontrivial critical point for $J_{\e}$. Moreover, the corresponding critical values coincide and 
\begin{equation*}
\inf_{u\in \mathbb{S}_{\e}^{+}} \psi_{\e}(u)= \inf_{u\in \mathcal{N}_{\e}} J_{\e}(u).  
\end{equation*}
\end{compactenum}
\end{prop}

\begin{remark}\label{rem3MAW}
As in \cite{SW}, we have the following variational characterization of the infimum of $J_{\e}$ over $\mathcal{N}_{\e}$:
\begin{align*}
0<c_{\e}&=\inf_{u\in \mathcal{N}_{\e}} J_{\e}(u)=\inf_{u\in \mathcal{H}_{\e}^{+}} \max_{t>0} J_{\e}(tu)=\inf_{u\in \mathbb{S}_{\e}^{+}} \max_{t>0} J_{\e}(tu).
\end{align*}
Moreover, if
\begin{align*}
c'_{\e}=\inf_{\gamma\in \Gamma_{\e}} \max_{t\in [0, 1]} J_{\e}(\gamma(t)), \, \mbox{ where } \, \Gamma_{\e}=\{\gamma\in C([0, 1], \h): \gamma(0)=0 \mbox{ and } J_{\e}(\gamma(1))<0\},
\end{align*}
then we can argue as in \cite{DF, Rab, W} to verify that $c_{\e}=c'_{\e}$.
\end{remark}

\noindent
The main feature of $J_{\e}$ is that  it satisfies the following compactness property:
\begin{lem}\label{PSc}
The functional $J_{\e}$ satisfies the $(PS)_{c}$ condition at any level $c\in \R$.
\end{lem}
\begin{proof}
Let $(u_{n})\subset \h$ be a Palais-Smale sequence for $J_{\e}$ at the level $c$, that is, as $n\ri \infty$ 
$$
J_{\e}(u_{n})\rightarrow c \mbox{ and } J_{\e}'(u_{n})\rightarrow 0.
$$
First, we show that $(u_{n})$ is bounded in $\h$. Indeed, using $(g_3)$, we get
\begin{align*}
C(1+\|u_{n}\|_{\e})&\geq J_{\e}(u_{n})-\frac{1}{\theta}\langle J'_{\e}(u_{n}), u_{n}\rangle \\
&= \left(\frac{1}{2}-\frac{1}{\theta}\right)\|u_{n}\|^{2}_{\e}+\frac{1}{\theta}\int_{\R^{N}} \left[g_{\e}(x, |u_{n}|^{2})|u_{n}|^{2}-\frac{\theta}{2} G_{\e}(x, |u_{n}|^{2})\right]\, dx\\
&\geq \left(\frac{1}{2}-\frac{1}{\theta}\right)\|u_{n}\|^{2}_{\e}
-\left(\frac{1}{2}-\frac{1}{\theta}\right)\int_{\Lambda^{c}_{\e}} G_{\e}(x, |u_{n}|^{2})\, dx \\
&\geq \left(\frac{1}{2}-\frac{1}{\theta}\right)\|u_{n}\|^{2}_{\e}-\left(\frac{1}{2}-\frac{1}{\theta}\right)\frac{1}{k}\int_{\Lambda^{c}_{\e}} V_{\e} (x) |u_{n}|^{2}\, dx \\
&\geq \left(\frac{\theta-2}{2\theta}\right)\left( 1-\frac{1}{k} \right)\|u_{n}\|^{2}_{\e},
\end{align*}
and recalling that $k>1$ and $\theta>2$, we deduce that $(u_{n})$ is bounded in $\h$. 
Since $\h$ is a reflexive space, we can find a subsequence, still denoted by $(u_{n})$, and $u\in \h$ such that 
\begin{align}\label{ADOMconvexp}
\begin{array}{ll}
u_{n}\rightharpoonup u \quad &\mbox{ in } \h \mbox{ as } n\ri \infty, \\
u_{n}\ri u \quad &\mbox{ in } L^{q}_{loc}(\R^N, \C) \mbox{ for all } q\in [1, \2) \mbox{ as } n\ri \infty, \\
|u_{n}|\ri |u| \quad &\mbox{ a.e. in } \R^N \mbox{ as } n\ri \infty.  
\end{array}
\end{align}
Using $(g_1)$ and $(g_2)$, we see that 
\begin{equation}\label{newformulaexp}
\lim_{n\ri \infty} \Re\int_{\R^N} g_{\e}(x, |u_{n}|^{2})u_{n}\overline{\phi} \, dx = \Re\int_{\R^N} g_{\e}(x, |u|^{2})u \overline{\phi} \, dx \quad \mbox{ for all } \phi \in C^{\infty}_{c}(\R^N, \C).
\end{equation}
Taking into account \eqref{ADOMconvexp}, \eqref{newformulaexp} and Theorem \ref{density}, we deduce that
\begin{align*}
\langle J_{\e}'(u), \phi \rangle =0 \quad \mbox{ for all } \phi \in \h, 
\end{align*}
that is $u$ is a critical point for $J_{\e}$.
In particular, $\langle J_{\e}'(u), u \rangle =0$, or equivalently
\begin{align}\label{new1exp}
[u]^{2}_{A_{\e}} + \int_{\Lambda_{\e}} V_{\e}(x) |u|^{2}\, dx+\int_{\Lambda^{c}_{\e}} \mathcal{C}(\e x, |u|^{2})\, dx= \int_{\Lambda_{\e}} f(|u|^{2})|u|^{2}\, dx,
\end{align}
where $\mathcal{C}(x, t)=V(x)t-g(x, t)t$. Note that, by $(g_3)$-$(ii)$, it holds
\begin{align}\label{LIMITETERMOZETA}
V(x) t\geq \mathcal{C}(x, t)\geq \left(1-\frac{1}{k}\right)V(x)t\geq 0 \quad\mbox{ for all } x\in \Lambda^{c}, \, t\geq 0.
\end{align}
Recalling that $\langle J_{\e}'(u_{n}), u_{n}\rangle =o_{n}(1)$, we also know that
\begin{align}\label{new2exp}
[u_{n}]^{2}_{A_{\e}} + \int_{\Lambda_{\e}} V_{\e}(x) |u_{n}|^{2}\, dx+\int_{\Lambda^{c}_{\e}} \mathcal{C}(\e x, |u_{n}|^{2})\, dx= \int_{\Lambda_{\e}} f(|u_{n}|^{2})|u_{n}|^{2}\, dx +o_{n}(1). 
\end{align}
Since $\Lambda_{\e}$ is bounded, it follows from the local compact embeddings in Theorem \ref{Sembedding} that 
\begin{align}\label{new22exp}
\lim_{n\ri \infty} \int_{\Lambda_{\e}} f(|u_{n}|^{2})|u_{n}|^{2} \,dx =  \int_{\Lambda_{\e}} f(|u|^{2})|u|^{2} \,dx,
\end{align}
and 
\begin{align}\label{new222exp}
\lim_{n\ri \infty} \int_{\Lambda_{\e}} V_{\e}(x) |u_{n}|^{2}\, dx=\int_{\Lambda_{\e}} V_{\e}(x) |u|^{2}\, dx.
\end{align}
Putting together \eqref{new1exp}, \eqref{new2exp}, \eqref{new22exp} and \eqref{new222exp}, we deduce that
\begin{align*}
\limsup_{n\ri \infty} &\left([u_{n}]^{2}_{A_{\e}} + \int_{\Lambda^{c}_{\e}} \mathcal{C}(\e x, |u_{n}|^{2})\, dx \right)=[u]^{2}_{A_{\e}} + \int_{\Lambda^{c}_{\e}} \mathcal{C}(\e x, |u|^{2})\, dx. 
\end{align*}
Now, by \eqref{LIMITETERMOZETA} and Fatou's lemma, we get
\begin{align*}
\liminf_{n\ri \infty} &\left([u_{n}]^{2}_{A_{\e}} + \int_{\Lambda^{c}_{\e}} \mathcal{C}(\e x, |u_{n}|^{2})\, dx \right)\geq [u]^{2}_{A_{\e}} + \int_{\Lambda^{c}_{\e}} \mathcal{C}(\e x, |u|^{2})\, dx. 
\end{align*}
Hence,
\begin{align}\label{new2222exp}
\lim_{n\ri \infty} [u_{n}]^{2}_{A_{\e}}=[u]^{2}_{A_{\e}},
\end{align}
and
\begin{align*}
\lim_{n\ri \infty} \int_{\Lambda^{c}_{\e}} \mathcal{C}(\e x, |u_{n}|^{2})\, dx =\int_{\Lambda^{c}_{\e}} \mathcal{C}(\e x, |u|^{2})\, dx. 
\end{align*}
The last limit, Fatou's lemma and \eqref{LIMITETERMOZETA}
lead to
\begin{align*}
\int_{\Lambda^{c}_{\e}} V_{\e}(x) |u|^{2}\, dx &\leq \liminf_{n\ri \infty}\int_{\Lambda^{c}_{\e}} V_{\e}(x) |u_{n}|^{2}\, dx \\
&\leq \limsup_{n\ri \infty}\int_{\Lambda^{c}_{\e}} V_{\e}(x) |u_{n}|^{2}\, dx \\
&\leq \left(\frac{k}{k-1}\right) \limsup_{n\ri \infty} \int_{\Lambda^{c}_{\e}} \mathcal{C}(\e x, |u_{n}|^{2})\, dx  \\
&=\left(\frac{k}{k-1}\right)\int_{\Lambda^{c}_{\e}} \mathcal{C}(\e x, |u|^{2})\, dx\leq \left(\frac{k}{k-1}\right) \int_{\Lambda^{c}_{\e}} V_{\e}(x) |u|^{2}\, dx \quad \mbox{ for all } k>1,
\end{align*}
and by sending $k\ri \infty$ we find
\begin{align*}
\lim_{n\ri \infty} \int_{\Lambda^{c}_{\e}} V_{\e}(x) |u_{n}|^{2}\, dx =\int_{\Lambda^{c}_{\e}} V_{\e}(x) |u|^{2}\, dx 
\end{align*}
which combined with \eqref{new222exp} gives 
\begin{align}\label{new22222exp}
\lim_{n\ri \infty} \int_{\R^N} V_{\e}(x) |u_{n}|^{2}\, dx =\int_{\R^N} V_{\e}(x) |u|^{2}\, dx. 
\end{align}
Putting together \eqref{new2222exp} and \eqref{new22222exp}, we obtain that
\begin{align*}
\lim_{n\ri \infty} \|u_{n}\|_{\e}^{2} = \|u\|_{\e}^{2}. 
\end{align*}
Since $\h$ is a Hilbert space and $u_{n}\rightharpoonup u$ in $\h$ as $n\ri \infty$, we conclude that $u_{n}\ri u$ in $\h$ as $n\ri \infty$.
\end{proof}
%

\begin{cor}\label{cor2.1MAW}
The functional $\psi_{\e}$ satisfies the $(PS)_{c}$ condition on $\mathbb{S}_{\e}^{+}$ at any level $c\in \R$. 
\end{cor}

\begin{proof}
Let $(u_{n})\subset \mathbb{S}_{\e}^{+}$ be a Palais-Smale sequence for $\psi_{\e}$ at the level $c$. Then, 
\begin{equation*}
\psi_{\e}(u_{n})\rightarrow c \quad \mbox{ and } \quad \|\psi_{\e}'(u_{n})\|_{*}\rightarrow 0,
\end{equation*}
where $\|\cdot \|_{*}$ denotes the norm in the dual space $(T_{u_{n}} \mathbb{S}_{\e}^{+})^{*}$.
It follows from Proposition \ref{prop2.1MAW}-$(c)$ that $(m_{\e}(u_{n}))\subset \h$ is a Palais-Smale sequence for $J_{\e}$ at the level $c$. Then, using Lemma \ref{PSc}, there exists $u\in \mathbb{S}_{\e}^{+}$ such that, up to a subsequence, 
\begin{equation*}
m_{\e}(u_{n})\rightarrow m_{\e}(u) \quad \mbox{ in } \h. 
\end{equation*}
Applying Lemma \ref{lem2.3MAW}-$(iii)$ we can infer that $u_{n}\rightarrow u$ in $\mathbb{S}_{\e}^{+}$.
\end{proof}

We end this section by showing the following existence result for \eqref{MPe}.
\begin{thm}
Assume that $(V_1)$-$(V_2)$ and $(f_1)$-$(f_4)$ hold. Then, for all $\e>0$, there exists a ground state solution to \eqref{MPe}.
\end{thm}
\begin{proof}
In view of Lemma \ref{MPG} and Lemma \ref{PSc}, we can apply the mountain pass theorem \cite{AR} to see that for all $\e>0$ there exists a nontrivial critical point $u_{\e}\in \h$ of $J_{\e}$. By Remark \ref{rem3MAW}, we deduce the thesis.
\end{proof}


\section{The limiting scalar problem}

In this section we focus our attention on the limiting scalar problem associated to \eqref{MPe}, namely
\begin{equation}\label{AP0}
\left\{
\begin{array}{ll}
(-\Delta)^{s} u + V_{0} u=  f(u^{2})u \mbox{ in } \R^{N}, \\
u\in H^{s}(\mathbb{R}^{N}, \mathbb{R}).
\end{array}
\right.
\end{equation}
The  corresponding energy functional $I_{0}: \mathcal{H}_{0}\rightarrow \R$ is given by
\begin{align*}
I_{0}(u)=\frac{1}{2}\|u\|^{2}_{0}-\int_{\R^{N}} \frac{1}{2}F(u^{2})\, dx,
\end{align*}
where $\mathcal{H}_{0}$ stands for the fractional Sobolev space $H^{s}(\R^{N}, \R)$ endowed with the norm 
$$
\|u\|_{0}=\sqrt{[u]^{2}+V_{0}\|u\|^{2}_{L^{2}(\R^{N})}}.
$$ 
For any $u, v\in \mathcal{H}_{0}$, we set
$$
\langle u, v\rangle_{0}=\iint_{\R^{2N}} \frac{(u(x)-u(y))(v(x)-v(y))}{|x-y|^{N+2s}}\, dxdy+V_{0}\int_{\R^{N}} u v\, dx.
$$
Let 
$$
\mathcal{M}_{0}=\{u\in  \mathcal{H}_{0}\setminus \{0\} : \langle I_{\mu}'(u), u \rangle =0  \}
$$
be the Nehari manifold associated with $I_{0}$. 
Let us consider
\begin{equation*}
\mathcal{H}_{0}^{+}= \{u\in \mathcal{H}_{0} : |{\rm supp}(|u|)|>0\}. 
\end{equation*}
Let $\mathbb{S}_{0}$ be the unit sphere of $\mathcal{H}_{0}$ and we denote by $\mathbb{S}_{0}^{+}= \mathbb{S}_{\e}\cap \mathcal{H}_{0}^{+}$.
We observe that $\mathbb{S}_{0}^{+}$ is an incomplete $C^{1,1}$-manifold of codimension $1$, modelled on $\mathcal{H}_{0}$ and contained in the open $\mathcal{H}_{0}^{+}$. Then, $\mathcal{H}_{0}=T_{u} \mathbb{S}_{0}^{+} \oplus \mathbb{R} u$ for each $u\in \mathbb{S}_{0}^{+}$, where
\begin{equation*}
T_{u} \mathbb{S}_{0}^{+}= \{v \in \mathcal{H}_{0} : \langle u, v\rangle_{0}=0\}.
\end{equation*} 
Arguing as in the proof of Lemma \ref{lem2.3MAW} (see also \cite{A1}) we have the following result.
\begin{lem}\label{lem2.3AMAW}
Assume that $(f_1)$-$(f_4)$ hold. Then, 
\begin{compactenum}[$(i)$]
\item For each $u\in \mathcal{H}_{0}^{+}$, let $h:\mathbb{R}_{+}\rightarrow \mathbb{R}$ be defined by $h_{u}(t)= I_{0}(tu)$. Then, there is a unique $t_{u}>0$ such that 
\begin{align*}
&h'_{u}(t)>0 \mbox{ in } (0, t_{u}),\\
&h'_{u}(t)<0 \mbox{ in } (t_{u}, \infty).
\end{align*}
\item There exists $\tau>0$ independent of $u$ such that $t_{u}\geq \tau$ for any $u\in \mathbb{S}_{0}^{+}$. Moreover, for each compact set $\mathbb{K}\subset \mathbb{S}_{0}^{+}$ there is a positive constant $C_{\mathbb{K}}$ such that $t_{u}\leq C_{\mathbb{K}}$ for any $u\in \mathbb{K}$.
\item The map $\hat{m}_{0}: \mathcal{H}_{0}^{+}\rightarrow \mathcal{M}_{0}$ given by $\hat{m}_{0}(u)= t_{u}u$ is continuous and $m_{0}= \hat{m}_{0}|_{\mathbb{S}_{0}^{+}}$ is a homeomorphism between $\mathbb{S}_{0}^{+}$ and $\mathcal{M}_{0}$. Moreover, $m_{0}^{-1}(u)=\frac{u}{\|u\|_{0}}$.
\item If there is a sequence $(u_{n})\subset \mathbb{S}_{0}^{+}$ such that ${\rm dist}(u_{n}, \partial \mathbb{S}_{0}^{+})\rightarrow 0$ then $\|m_{0}(u_{n})\|_{0}\rightarrow \infty$ and $I_{0}(m_{0}(u_{n}))\rightarrow \infty$.
\end{compactenum}
\end{lem}

Let us define the maps 
\begin{equation*}
\hat{\psi}_{0}: \mathcal{H}_{0}^{+} \rightarrow \mathbb{R} \quad \mbox{ and } \quad \psi_{0}: \mathbb{S}_{0}^{+}\rightarrow \mathbb{R}, 
\end{equation*}
by $\hat{\psi}_{0}(u)= I_{0}(\hat{m}_{0}(u))$ and $\psi_{0}=\hat{\psi}_{0}|_{\mathbb{S}_{0}^{+}}$. 

The next result is a direct consequence of Lemma \ref{lem2.3AMAW} and Corollary $10$ in \cite{SW}. 
\begin{prop}\label{prop2.1AMAW}
Assume that  $(f_{1})$-$(f_{4})$ hold. Then, 
\begin{compactenum}[$(a)$]
\item $\hat{\psi}_{0} \in C^{1}(\mathcal{H}_{0}^{+}, \mathbb{R})$ and 
\begin{equation*}
\langle \hat{\psi}_{0}'(u), v\rangle = \frac{\|\hat{m}_{0}(u)\|_{0}}{\|u\|_{0}} \langle I_{0}'(\hat{m}_{0}(u)), v\rangle, 
\end{equation*}
for every $u\in \mathcal{H}_{0}^{+}$ and $v\in \mathcal{H}_{0}$. 
\item $\psi_{0} \in C^{1}(\mathbb{S}_{0}^{+}, \mathbb{R})$ and 
\begin{equation*}
\langle \psi_{0}'(u), v \rangle = \|m_{0}(u)\|_{0} \langle I_{0}'(m_{0}(u)), v\rangle, 
\end{equation*}
for every $v\in T_{u}\mathbb{S}_{0}^{+}$.
\item If $(u_{n})$ is a Palais-Smale sequence for $\psi_{0}$, then $(m_{0}(u_{n}))$ is a Palais-Smale sequence for $I_{0}$. If $(u_{n})\subset \mathcal{M}_{0}$ is a bounded Palais-Smale sequence for $I_{0}$, then $(m_{0}^{-1}(u_{n}))$ is a Palais-Smale sequence for  $\psi_{0}$. 
\item $u$ is a critical point of $\psi_{0}$ if and only if $m_{0}(u)$ is a nontrivial critical point for $I_{0}$. Moreover, the corresponding critical values coincide and 
\begin{equation*}
\inf_{u\in \mathbb{S}_{0}^{+}} \psi_{0}(u)= \inf_{u\in \mathcal{M}_{0}} I_{0}(u).  
\end{equation*}
\end{compactenum}
\end{prop}

\begin{remark}\label{rem3AMAW}
As in \cite{SW}, we have the following variational characterization of the infimum of $I_{0}$ over $\mathcal{M}_{0}$:
\begin{align*}
0<d_{0}&=\inf_{u\in \mathcal{M}_{0}} I_{0}(u)=\inf_{u\in \mathcal{H}_{0}^{+}} \max_{t>0} I_{0}(tu)=\inf_{u\in \mathbb{S}_{0}^{+}} \max_{t>0} I_{0}(tu).
\end{align*}
\end{remark}

\noindent
The next lemma allows us to assume that the weak limit of a Palais-Smale sequence is nontrivial.
\begin{lem}\label{LionsFS}
Let $(u_{n})\subset \mathcal{H}_{0}$ be a Palais-Smale sequence for $I_{0}$ at the level $d_{0}$. Then, one and only one of the following alternatives holds:
\begin{compactenum}[$(i)$]
\item $u_{n}\rightarrow 0$ in $\mathcal{H}_{0}$, 
\item there exist a sequence $(y_{n})\subset \R^{N}$ and constants $R, \beta>0$ such that
$$
\liminf_{n\rightarrow \infty} \int_{B_{R}(y_{n})} |u_{n}|^{2}dx\geq \beta>0.
$$
\end{compactenum}
\end{lem}
\begin{proof}
Assume that $(ii)$ does not occur. Arguing as in the proof of Lemma \ref{PSc} we see that $(u_{n})$ is bounded in $\mathcal{H}_{0}$. Then we use Lemma \ref{Lions} to deduce that $u_{n}\rightarrow 0$ in $L^{r}(\R^{N}, \R)$ for all $r\in (2, \2)$. In view of $(f_1)$-$(f_2)$, we get $\int_{\R^{N}} f(u^{2}_{n})u^{2}_{n}dx=o_{n}(1)$. This fact combined with $\langle I'_{0}(u_{n}), u_{n}\rangle=o_{n}(1)$ yields $\|u_{n}\|^{2}_{0}=o_{n}(1)$. 
\end{proof}

\begin{remark}\label{rem5MAW}
As it has been mentioned, if $u$ is the weak limit of a Palais-Smale sequence $(u_{n})\subset \mathcal{H}_{0}$ of $I_{0}$ at the level $d_{0}$, then we can assume that $u\neq 0$. Otherwise, we would have $u_{n}\rightharpoonup 0$ in $\mathcal{H}_{0}$ and, if $u_{n}\nrightarrow 0$ in $\mathcal{H}_{0}$, we deduce from Lemma \ref{LionsFS} that there are $(y_{n})\subset \mathbb{R}^{N}$ and $R, \beta>0$ such that
\begin{equation*}
\liminf_{n\rightarrow \infty} \int_{B_{R}(y_{n})} u_{n}^{2} \,dx \geq \beta >0. 
\end{equation*}
Set $v_{n}(x)=u_{n}(x+y_{n})$. Then we see that $(v_{n})$ is a Palais-Smale sequence for $I_{0}$ at the level $d_{0}$, $(v_{n})$ is bounded in $\mathcal{H}_{0}$ and there exists $v\in \mathcal{H}_{0}$ such that $v_{n}\rightharpoonup v$ in $\mathcal{H}_{0}$ with $v\neq 0$.
\end{remark}

\noindent 
Next we prove an existence result for the scalar autonomous problem \eqref{AP0}.
\begin{thm}\label{FS}
Let $(u_{n})\subset \mathcal{H}_{0}$ be a Palais-Smale sequence of $I_{0}$ at the level $d_{0}$. Then there exists $u\in \mathcal{H}_{0}\setminus \{0\}$, with $u\geq 0$, such that, up to a subsequence, we have $u_{n}\ri u$ in $\mathcal{H}_{0}$. Moreover, $u$ is a positive ground state for \eqref{AP0}.
\end{thm}
\begin{proof}
Arguing as in the proof of Lemma \ref{PSc}, we can see that $(u_{n})$ is bounded in $\mathcal{H}_{0}$. Therefore, up to going a subsequence, we may assume that $u_{n}\rightharpoonup u$ in $\mathcal{H}_{0}$. Standard arguments show that $u$ is a critical point of $I_{0}$. In light of Remark \ref{rem5MAW}, we may assume that $u$ is not trivial. 
Hence, $u\in \M_{0}$. Now we prove that $I_{0}(u)=d_{0}$. Indeed, by $u\in \M_{0}$, $(f_3)$ and Fatou's Lemma, we have
\begin{align*}
d_{0}&\leq I_{0}(u)-\frac{1}{\theta}\langle I'_{0}(u), u\rangle\\
&\leq \liminf_{n\rightarrow \infty} \left[\left(\frac{1}{2}-\frac{1}{\theta}\right)\|u_{n}\|_{0}^{2}+\int_{\R^{N}} \frac{1}{\theta} f(u^{2}_{n})u^{2}_{n}- \frac{1}{2} F(u^{2}_{n}) \,dx\right] \\
&=\liminf_{n\rightarrow \infty} \left[I_{0}(u_{n})-\frac{1}{\theta}\langle I'_{0}(u_{n}), u_{n}\rangle\right] \\
&=d_{0}.
\end{align*}
Since $\langle I'_{0}(u), u^{-}\rangle=0$, where $u^{-}=\min\{u, 0\}$, and using $(f_1)$, we can see that $u\geq 0$ in $\R^{N}$. Performing a standard Moser iteration argument (see Proposition $5.1.1$ in \cite{DMV}) we see that $u\in L^{\infty}(\R^{N}, \R)$ and applying Proposition $2.9$ in \cite{Silvestre} we infer that $u\in C^{0, \alpha}(\R^{N}, \R)$ for some $\alpha\in (0, 1)$. By the strong maximum principle (see Theorem $1.4$ in \cite{DPQJDE}), we deduce that $u>0$ in $\R^{N}$.
\end{proof}

The next lemma is a compactness result for the autonomous problem which will be used later.
\begin{lem}\label{lem3.3MAW}
Let $(u_{n})\subset \mathcal{M}_{0}$ be a sequence such that $I_{0}(u_{n})\rightarrow d_{0}$. Then  
$(u_{n})$ has a convergent subsequence in $\mathcal{H}_{0}$.
\end{lem}
\begin{proof}
Since $(u_{n})\subset \mathcal{M}_{0}$ and $I_{0}(u_{n})\rightarrow d_{0}$, we can apply Lemma \ref{lem2.3AMAW}-$(iii)$, Proposition \ref{prop2.1AMAW}-$(d)$ and the definition of $d_{0}$ to infer that
$$
\nu_{n}=m^{-1}_{0}(u_{n})=\frac{u_{n}}{\|u_{n}\|_{0}}\in \mathbb{S}_{0}^{+}
$$
and
$$
\psi_{0}(\nu_{n})=I_{0}(u_{n})\rightarrow d_{0}=\inf_{v\in \mathbb{S}_{0}^{+}}\psi_{0}(v).
$$
Let us introduce the following map $\mathcal{F}: \overline{\mathbb{S}}_{0}^{+}\rightarrow \mathbb{R}\cup \{\infty\}$ defined by setting
$$
\mathcal{F}(u)=
\begin{cases}
\psi_{0}(u)& \text{ if $u\in \mathbb{S}_{0}^{+}$}, \\
\infty   & \text{ if $u\in \partial \mathbb{S}_{0}^{+}$}.
\end{cases}
$$ 
We note that 
\begin{itemize}
\item $(\overline{\mathbb{S}}_{0}^{+}, \delta_{0})$, where $\delta_{0}(u, v)=\|u-v\|_{0}$, is a complete metric space;
\item $\mathcal{F}\in C(\overline{\mathbb{S}}_{0}^{+}, \mathbb{R}\cup \{\infty\})$, by Lemma \ref{lem2.3AMAW}-$(iv)$;
\item $\mathcal{F}$ is bounded below, by Proposition \ref{prop2.1AMAW}-$(d)$.
\end{itemize}
By applying Ekeland's variational principle \cite{Ek}, we can find a Palais-Smale sequence $(\hat{v}_{n})\subset \mathbb{S}_{0}^{+}$ for $\psi_{0}$ at the level $d_{0}$ and such that $\|\hat{v}_{n}-\nu_{n}\|_{0}=o_{n}(1)$.
Now the remainder of the proof follows from Proposition \ref{prop2.1AMAW}, Theorem \ref{FS} and arguing as in the proof of Corollary \ref{cor2.1MAW}. 
\end{proof}

\noindent
Finally, we prove the following interesting relation between $c_{\e}$ and $d_{0}$.
\begin{lem}\label{AMlem1}
The numbers $c_{\e}$ and $d_{0}$ satisfy the following inequality:
$$
\limsup_{\e\rightarrow 0} c_{\e}\leq d_{0}.
$$
\end{lem}
\begin{proof}
Let $w\in H^{s}(\R^{N}, \R)$ be a positive ground state to \eqref{AP0}, so that $I'_{0}(w)=0$ and $I_{0}(w)=d_{0}$, and let $\eta\in C^{\infty}_{c}(\R^{N}, [0,1])$ be a cut-off function such that $\eta=1$ in $B_{\frac{\delta}{2}}$ and $\supp(\eta)\subset B_{\delta}\subset \Lambda$ for some $\delta>0$. 
We recall that the existence of $w$ is guaranteed by Theorem \ref{FS}. Moreover, from Theorem $1.5$ in \cite{FQT}, we know that $w$ 
satisfies the following power-type decay estimate: 
\begin{equation}\label{remdecay}
0<w(x)\leq \frac{C}{|x|^{N+2s}} \quad \mbox{ for all } |x|>1.
\end{equation}

Let us define $w_{\e}(x)=\eta_{\e}(x)w(x) e^{\imath A(0)\cdot x}$, with $\eta_{\e}(x)=\eta(\e x)$ for $\e>0$, and we observe that $|w_{\e}|=\eta_{\e}w$ and $w_{\e}\in \h$ in light of Lemma \ref{aux}. Next we claim that
\begin{equation}\label{limwr}
\lim_{\e\rightarrow 0}\|w_{\e}\|^{2}_{\e}=\|w\|_{0}^{2}\in(0, \infty).
\end{equation}
Clearly, $\int_{\R^{N}} V_{\e}(x)|w_{\e}|^{2}dx\rightarrow \int_{\R^{N}} V_{0} |w|^{2}dx$. Then, it remains to show that
\begin{equation}\label{limwr*}
\lim_{\e\rightarrow 0}[w_{\e}]^{2}_{A_{\e}}=[w]^{2}.
\end{equation}
Using Lemma $5$ in \cite{PP}, we know that 
\begin{equation}\label{PPlem}
[\eta_{\e} w]\rightarrow [w] \mbox{ as } \e\rightarrow 0.
\end{equation}
On the other hand
\begin{align*}
[w_{\e}]_{A_{\e}}^{2}
&=\iint_{\R^{2N}} \frac{|e^{\imath A(0)\cdot x}\eta_{\e}(x)w(x)-e^{\imath A_{\e}(\frac{x+y}{2})\cdot (x-y)}e^{\imath A(0)\cdot y} \eta_{\e}(y)w(y)|^{2}}{|x-y|^{N+2s}} dx dy \nonumber \\
&=[\eta_{\e} w]^{2}
+\iint_{\R^{2N}} \frac{\eta_{\e}^2(y)w^2(y) |e^{\imath [A_{\e}(\frac{x+y}{2})-A(0)]\cdot (x-y)}-1|^{2}}{|x-y|^{N+2s}} dx dy\\
&\quad+2\Re \iint_{\R^{2N}} \frac{(\eta_{\e}(x)w(x)-\eta_{\e}(y)w(y))\eta_{\e}(y)w(y)(1-e^{-\imath [A_{\e}(\frac{x+y}{2})-A(0)]\cdot (x-y)})}{|x-y|^{N+2s}} dx dy \\
&= [\eta_{\e} w]^{2}+X_{\e}+2Y_{\e}.
\end{align*}
Then, in view of 
$|Y_{\e}|\leq [\eta_{\e} w] \sqrt{X_{\e}}$ and \eqref{PPlem}, it is suffices to prove that $X_{\e}\rightarrow 0$ as $\e\rightarrow 0$ to deduce that \eqref{limwr*} is satisfied.
Let us note that for $0<\beta<\alpha/({1+\alpha-s})$, 
\begin{equation}\label{Ye}
\begin{split}
X_{\e}
&\leq \int_{\R^{N}} w^{2}(y) dy \int_{|x-y|\geq\e^{-\beta}} \frac{|e^{\imath [A_{\e}(\frac{x+y}{2})-A(0)]\cdot (x-y)}-1|^{2}}{|x-y|^{N+2s}} dx\\
&\quad +\int_{\R^{N}} w^{2}(y) dy  \int_{|x-y|<\e^{-\beta}} \frac{|e^{\imath [A_{\e}(\frac{x+y}{2})-A(0)]\cdot (x-y)}-1|^{2}}{|x-y|^{N+2s}} dx\\
&=X^{1}_{\e}+X^{2}_{\e}.
\end{split}
\end{equation}
Using $|e^{\imath t}-1|^{2}\leq 4$ and $w\in H^{s}(\R^{N}, \R)$, we get
\begin{equation}\label{Ye1}
X_{\e}^{1}\leq C \int_{\R^{N}} w^{2}(y) dy \int_{\e^{-\beta}}^\infty \rho^{-1-2s} d\rho\leq C \e^{2\beta s} \rightarrow 0.
\end{equation}
Since $|e^{\imath t}-1|^{2}\leq t^{2}$ for all $t\in \R$, $A\in C^{0,\alpha}(\R^N,\R^N)$ with $\alpha\in(0,1]$, and $|x+y|^{2}\leq 2(|x-y|^{2}+4|y|^{2})$, we have
\begin{equation}\label{Ye2}
\begin{split}
X^{2}_{\e}&
	\leq \int_{\R^{N}} w^{2}(y) dy  \int_{|x-y|<\e^{-\beta}} \frac{|A_{\e}\left(\frac{x+y}{2}\right)-A(0)|^{2} }{|x-y|^{N+2s-2}} dx \\
	&\leq C\e^{2\alpha} \int_{\R^{N}} w^{2}(y) dy  \int_{|x-y|<\e^{-\beta}} \frac{|x+y|^{2\alpha} }{|x-y|^{N+2s-2}} dx \\
	&\leq C\e^{2\alpha} \left(\int_{\R^{N}} w^{2}(y) dy  \int_{|x-y|<\e^{-\beta}} \frac{1 }{|x-y|^{N+2s-2-2\alpha}} dx\right.\\
	&\qquad \qquad+ \left. \int_{\R^{N}} |y|^{2\alpha} w^{2}(y) dy  \int_{|x-y|<\e^{-\beta}} \frac{1}{|x-y|^{N+2s-2}} dx\right) \\
	&= C\e^{2\alpha} (X^{2, 1}_{\e}+X^{2, 2}_{\e}).
	\end{split}
	\end{equation}	
	Then
	\begin{equation}\label{Ye21}
	X^{2, 1}_{\e}
	= C  \int_{\R^{N}} w^{2}(y) dy \int_0^{\e^{-\beta}} \rho^{1+2\alpha-2s} d\rho
	\leq C\e^{-2\beta(1+\alpha-s)}.
	\end{equation}
	On the other hand, using \eqref{remdecay}, we infer that
	\begin{equation}\label{Ye22}
	\begin{split}
	 X^{2, 2}_{\e}
	 &\leq C  \int_{\R^{N}} |y|^{2\alpha} w^{2}(y) dy \int_0^{\e^{-\beta}}\rho^{1-2s} d\rho  \\
	&\leq C \e^{-2\beta(1-s)} \left[\int_{B_1}  w^{2}(y) dy + \int_{\R^{N}\setminus B_{1}} \frac{1}{|y|^{2(N+2s)-2\alpha}} dy \right]  \\
	&\leq C \e^{-2\beta(1-s)}.
	\end{split}
	\end{equation}
	Taking into account \eqref{Ye}, \eqref{Ye1}, \eqref{Ye2}, \eqref{Ye21} and \eqref{Ye22} we can conclude that $X_{\e}\rightarrow 0$. Therefore \eqref{limwr} holds.
Now, let $t_{\e}>0$ be the unique number such that 
\begin{equation*}
J_{\e}(t_{\e} w_{\e})=\max_{t\geq 0} J_{\e}(t w_{\e}).
\end{equation*}
Then $t_{\e}$ satisfies 
\begin{equation}\label{AS1}
\|w_{\e}\|_{\e}^{2}=\int_{\R^{N}} g_{\e}(x, t_{\e}^{2} |w_{\e}|^{2}) |w_{\e}|^{2}dx=\int_{\R^{N}} f(t_{\e}^{2} |w_{\e}|^{2}) |w_{\e}|^{2}dx
\end{equation}
where we used $\supp(\eta)\subset \Lambda$ and $g=f$ on $\Lambda\times \R$.
Let us prove that $t_{\e}\rightarrow 1$ as $\e\rightarrow 0$. Using $\eta=1$ in $B_{\frac{\delta}{2}}$ and $(f_4)$ we find
$$
\|w_{\e}\|_{\e}^{2}\geq f(t_{\e}^{2}\alpha^{2}_{0})\int_{B_{\frac{\delta}{2}}}|w|^{2}\,dx, 
$$
where $\alpha_{0}=\min_{\overline{B}_{\delta/2}} w>0$ (remark that $w$ is a continuous positive function). So, if $t_{\e}\rightarrow \infty$ as $\e\rightarrow 0$, then we can use $(f_3)$ to deduce that $\|w\|_{0}^{2}= \infty$, which gives a contradiction by \eqref{limwr}.
On the other hand, if $t_{\e}\rightarrow 0$ as $\e\rightarrow 0$, we can exploit the growth assumptions on $f$ and \eqref{AS1} to infer that $\|w\|_{0}^{2}= 0$ which is in contrast with \eqref{limwr}.
In conclusion, $t_{\e}\rightarrow t_{0}\in (0, \infty)$ as $\e\rightarrow 0$.
Now, taking the limit as $\e\rightarrow 0$ in \eqref{AS1} and using \eqref{limwr}, we can see that 
\begin{equation}\label{AS2}
\|w\|_{0}^{2}=\int_{\R^{N}} f(t_{0}^{2}w^{2}) w^{2}\,dx.
\end{equation}
In view of $w\in \mathcal{M}_{0}$ and $(f_4)$, we have that $t_{0}=1$. Then, by \eqref{limwr}, $t_{\e}\rightarrow 1$ and applying the dominated convergence theorem, we obtain that $\lim_{\e\rightarrow 0} J_{\e}(t_{\e} w_{\e})=I_{0}(w)=d_{0}$.
Since $c_{\e}\leq \max_{t\geq 0} J_{\e}(t w_{\e})=J_{\e}(t_{\e} w_{\e})$, we can conclude  that
$\limsup_{\e\rightarrow 0} c_{\e}\leq d_{0}$.
\end{proof}



\section{Multiplicity result for the modified problem}
In this section we make use of the Ljusternik-Schnirelman category theory to obtain multiple solutions to \eqref{MPe}. 
In particular, we relate the number of positive solutions of \eqref{MPe} to the topology of the set $M$. To do this, we introduce some useful tools needed to implement the barycenter machinery below.
Let $\delta>0$ be such that
$$
M_{\delta}=\{x\in \R^{N}: {\rm dist}(x, M)\leq \delta\}\subset \Lambda,
$$
and consider a smooth nonincreasing function $\eta:[0, \infty)\rightarrow \R$ such that $\eta(t)=1$ if $0\leq t\leq \frac{\delta}{2}$, $\eta(t)=0$ if $t\geq \delta$, $0\leq \eta\leq 1$ and $|\eta'(t)|\leq c$ for some $c>0$.
For any $y\in M$, we introduce 
$$
\Psi_{\e, y}(x)=\eta(|\e x-y|) w\left(\frac{\e x-y}{\e}\right)e^{\imath \tau_{y} \left( \frac{\e x-y}{\e} \right)},
$$
where $\tau_{y}(x)=\sum_{j=1}^{N} A_{j}(y)x_{j}$ and $w\in \mathcal{H}_{0}$ is a positive ground state solution to the autonomous problem \eqref{AP0} whose existence is guaranteed by Theorem \ref{FS}. Let $t_{\e}>0$ be the unique number such that 
$$
J_{\e}(t_{\e} \Psi_{\e, y})=\max_{t\geq 0} J_{\e}(t \Psi_{\e, y}). 
$$
Finally, we consider $\Phi_{\e}: M\rightarrow \N_{\e}$ defined by setting
$$
\Phi_{\e}(y)= t_{\e} \Psi_{\e, y}.
$$
\begin{lem}\label{lem3.4}
The functional $\Phi_{\e}$ satisfies the following limit
\begin{equation*}
\lim_{\e\rightarrow 0} J_{\e}(\Phi_{\e}(y))=d_{0} \quad \mbox{ uniformly in } y\in M.
\end{equation*}
\end{lem}
\begin{proof}
Assume by contradiction that there exist $\delta_{0}>0$, $(y_{n})\subset M$ and $\e_{n}\rightarrow 0$ such that 
\begin{equation}\label{puac}
|J_{\e_{n}}(\Phi_{\e_{n}}(y_{n}))-d_{0}|\geq \delta_{0}.
\end{equation}
Applying Lemma $4.1$ in \cite{AD} and the dominated convergence theorem, we see that 
\begin{align}\begin{split}\label{nio3}
&\| \Psi_{\e_{n}, y_{n}} \|^{2}_{\e_{n}}\rightarrow \|w\|^{2}_{0}\in (0, \infty). 
\end{split}\end{align}
On the other hand, since $\langle J'_{\e_{n}}(\Phi_{\e_{n}}(y_{n})),\Phi_{\e_{n}}(y_{n})\rangle=0$ and using the change of variable $z=\frac{\e_{n}x-y_{n}}{\e_{n}}$, it follows that
\begin{align*}
t_{\e_{n}}^{2}\|\Psi_{\e_{n}, y_{n}}\|_{\e_{n}}^{2}=\int_{\R^{N}} g(\e_{n}z+y_{n}, |t_{\e_{n}}\eta(|\e_{n}z|)w(z)|^{2}) |t_{\e_{n}}\eta(|\e_{n}z|)w(z)|^{2} dz.
\end{align*}
If $z\in B_{\frac{\delta}{\e_{n}}}$, then $\e_{n} z+y_{n}\in B_{\delta}(y_{n})\subset M_{\delta}\subset \Lambda$. Since $g(x,t)=f(t)$ for $(x, t)\in \Lambda\times [0, \infty)$, we have
\begin{align}\label{1nio}
t_{\e_{n}}^{2}\|\Psi_{\e_{n}, y_{n}}\|_{\e_{n}}^{2}=\int_{\R^{N}} f(|t_{\e_{n}}\eta(|\e_{n}z|)w(z)|^{2}) (t_{\e_{n}}\eta(|\e_{n}z|)w(z))^{2} \, dz.
\end{align}
In view of $\eta(|x|)=1$ for $x\in B_{\frac{\delta}{2}}$ and that $B_{\frac{\delta}{2}}\subset B_{\frac{\delta}{2\e_{n}}}$ for all $n$ large enough, it follows from \eqref{1nio} and $(f_4)$ that
\begin{align}\label{nioo}
\|\Psi_{\e_{n}, y_{n}}\|_{\e_{n}}^{2} &\geq \int_{B_{\frac{\delta}{2}}} \frac{f(|t_{\e_{n}}w(z)|^{2}) (t_{\e_{n}}w(z))^{2}}{t^{2}_{\e_{n}}}dz \nonumber \\
&\geq  f(|t_{\e_{n}}w(\hat{z})|^{2}) \int_{B_{\frac{\delta}{2}}} w^{2}(z)\,dz, 
\end{align}
where
\begin{equation*}
w(\hat{z})=\min_{z\in \overline{B}_{\frac{\delta}{2}}} w(z)>0.
\end{equation*} 
Now, assume by contradiction that $t_{\e_{n}}\rightarrow \infty$. 
This fact, \eqref{nio3} and \eqref{nioo} yield 
$$
\|w\|^{2}_{0}=\infty,
$$
that is a contradiction.
Hence, $(t_{\e_{n}})$ is bounded and, up to subsequence, we may assume that $t_{\e_{n}}\rightarrow t_{0}$ for some $t_{0}\geq 0$.  
In particular, $t_{0}>0$. In fact, if $t_{0}=0$, we see that \eqref{uNr} and \eqref{1nio} imply that
$$
r\leq \int_{\R^{3}} f(|t_{\e_{n}}\eta(|\e_{n}z|)w(z)|^{2}) (t_{\e_{n}}\eta(|\e_{n}z|)w(z))^{2}\, dz.
$$
Using $(f_1)$, $(f_2)$, \eqref{nio3} and the above inequality, we deduce that $t_{0}>0$.
Thus, letting $n\rightarrow \infty$ in \eqref{1nio}, we have that
\begin{align*}
\|w\|^{2}_{0}=\int_{\R^{N}} f((t_{0} w)^{2}) w^{2} \, dx.
\end{align*}
Bearing in mind that $w\in \M_{0}$ and using $(f_4)$, we infer that $t_{0}=1$.
Passing to the limit as $n\rightarrow \infty$ and using $t_{\e_{n}}\rightarrow 1$ we conclude that
$$
\lim_{n\rightarrow \infty} J_{\e_{n}}(\Phi_{\e_{n}, y_{n}})=I_{0}(w)=d_{0},
$$
which is in contrast with \eqref{puac}.
\end{proof}

\noindent
Let us fix $\rho=\rho(\delta)>0$ satisfying $M_{\delta}\subset B_{\rho}$, and we consider $\varUpsilon: \R^{N}\rightarrow \R^{N}$ given by
\begin{equation*}
\varUpsilon(x)=
\left\{
\begin{array}{ll}
x &\mbox{ if } |x|<\rho, \\
\frac{\rho x}{|x|} &\mbox{ if } |x|\geq \rho.
\end{array}
\right.
\end{equation*}
Then we define the barycenter map $\beta_{\e}: \N_{\e}\rightarrow \R^{N}$ as follows
\begin{align*}
\beta_{\e}(u)=\frac{\displaystyle{\int_{\R^{N}} \varUpsilon(\e x)|u(x)|^{2} \,dx}}{\displaystyle{\int_{\R^{N}} |u(x)|^{2} \,dx}}.
\end{align*}

\noindent
\begin{lem}\label{lem3.5N}
The following limit holds true:
\begin{equation*}
\lim_{\e \rightarrow 0} \beta_{\e}(\Phi_{\e}(y))=y \quad \mbox{ uniformly in } y\in M.
\end{equation*}
\end{lem}
\begin{proof}
Suppose by contradiction that there exist $\delta_{0}>0$, $(y_{n})\subset M$ and $\e_{n}\rightarrow 0$ such that
\begin{equation}\label{4.4}
|\beta_{\e_{n}}(\Phi_{\e_{n}}(y_{n}))-y_{n}|\geq \delta_{0}.
\end{equation}
Using the definitions of $\Phi_{\e_{n}}(y_{n})$, $\beta_{\e_{n}}$, $\psi$ and the change of variable $z= \frac{\e_{n} x-y_{n}}{\e_{n}}$, we can see that
$$
\beta_{\e_{n}}(\Phi_{\e_{n}}(y_{n}))=y_{n}+\frac{\int_{\R^{N}}[\varUpsilon(\e_{n}z+y_{n})-y_{n}] |\eta(|\e_{n}z|) w(z)|^{2}\, dz}{\int_{\R^{N}} |\eta(|\e_{n}z|) w(z)|^{2}\, dz}.
$$
Since $(y_{n})\subset M\subset B_{\rho}$, it follows from the dominated convergence theorem that
$$
|\beta_{\e_{n}}(\Phi_{\e_{n}}(y_{n}))-y_{n}|=o_{n}(1)
$$
which is in contrast with (\ref{4.4}).
\end{proof}

\noindent
The next compactness result will play a fundamental role to prove that the solutions of \eqref{MPe} are also solution to \eqref{R}.
\begin{lem}\label{prop3.3}
Let $\e_{n}\rightarrow 0$ and $(u_{n})=(u_{\e_{n}})\subset \mathcal{N}_{\e_{n}}$ be such that $J_{\e_{n}}(u_{n})\rightarrow d_{0}$. Then there exists $(\tilde{y}_{n})\subset \R^{N}$ such that $v_{n}(x)=|u_{n}|(x+\tilde{y}_{n})$ has a convergent subsequence in $\mathcal{H}_{0}$. Moreover, up to a subsequence, $y_{n}=\e_{n} \tilde{y}_{n}\rightarrow y_{0}$ for some $y_{0}\in M$.
\end{lem}
\begin{proof}
Taking into account $\langle J'_{\e_{n}}(u_{n}), u_{n}\rangle=0$, $J_{\e_{n}}(u_{n})\ri d_{0}$, it is easy to see that $(u_{n})$ is bounded in $\mathcal{H}_{\e_{n}}$. 
Then, there exists $C>0$ (independent of $n$) such that $\|u_{n}\|_{\e_{n}}\leq C$ for all $n\in \mathbb{N}$. Moreover, from Lemma \ref{DI}, we also know that $(|u_{n}|)$ is bounded in $H^{s}(\R^{N}, \R)$.\\
Now we prove that there exist a sequence $(\tilde{y}_{n})\subset \R^{N}$ and constants $R, \gamma>0$ such that
\begin{equation}\label{sacchi}
\liminf_{n\rightarrow \infty}\int_{B_{R}(\tilde{y}_{n})} |u_{n}|^{2} \, dx\geq \gamma>0.
\end{equation}
If by contradiction \eqref{sacchi} does not hold, then for all $R>0$ we get
$$
\lim_{n\rightarrow \infty}\sup_{y\in \R^{N}}\int_{B_{R}(y)} |u_{n}|^{2} \, dx=0.
$$
From the boundedness of $(|u_{n}|)$ in $H^{s}(\R^{N}, \R)$ and Lemma \ref{Lions}, we can see that $|u_{n}|\rightarrow 0$ in $L^{q}(\R^{N}, \R)$ for any $q\in (2, 2^{*}_{s})$. 
This fact together with \eqref{BERTSCH}
and the boundedness of $(|u_{n}|)$ in $L^{2}(\R^{N}, \R)$ yields that
\begin{align}\label{glimiti}
\lim_{n\rightarrow \infty}\int_{\R^{N}} g_{\e_{n}} (x, |u_{n}|^{2}) |u_{n}|^{2} \,dx=0= \lim_{n\rightarrow \infty}\int_{\R^{N}} G_{\e_{n}}(x, |u_{n}|^{2}) \, dx.
\end{align}
Taking into account $\langle J'_{\e_{n}}(u_{n}), u_{n}\rangle=0$ and \eqref{glimiti}, we can  infer that $\|u_{n}\|_{\e_{n}}\rightarrow 0$ as $n\rightarrow \infty$ and this implies that $J_{\e_{n}}(u_{n})\ri 0$ which is in contrast with $d_{0}>0$.
Now, we set $v_{n}(x)=|u_{n}|(x+\tilde{y}_{n})$. Then $(v_{n})$ is bounded in $H^{s}(\R^{N}, \R)$, and we may assume that 
$v_{n}\rightharpoonup v\not\equiv 0$ in $H^{s}(\R^{N}, \R)$  as $n\rightarrow \infty$.
Fix $t_{n}>0$ such that $\tilde{v}_{n}=t_{n} v_{n}\in \mathcal{M}_{0}$. By Lemma \ref{DI}, $(u_{n})\subset \mathcal{N}_{\e_{n}}$ and $J_{\e_{n}}(u_{n})\ri d_{0}$, we can see that 
$$
d_{0}\leq I_{0}(\tilde{v}_{n})\leq \max_{t\geq 0}J_{\e_{n}}(tu_{n})= J_{\e_{n}}(u_{n})=d_{0}+o_{n}(1)
$$
which implies that $I_{0}(\tilde{v}_{n})\rightarrow d_{0}$. In particular, $(\tilde{v}_{n})$ is bounded in $H^{s}(\R^{N}, \R)$ and $\tilde{v}_{n}\rightharpoonup \tilde{v}$ in $H^{s}(\R^{N}, \R)$.
Since $v_{n}\nrightarrow 0$  in $H^{s}(\R^{N}, \R)$ and $(\tilde{v}_{n})$ is bounded in $H^{s}(\R^{N}, \R)$, we deduce that $(t_{n})$ is bounded in $\R$ and, up to a subsequence, we can assume that $t_{n}\rightarrow t^{*}\geq 0$. Indeed $t^{*}>0$ due to the boundedness of $(v_{n})$ in $H^{s}(\R^{N}, \R)$,  and the fact that $\tilde{v}_{n}\nrightarrow 0$  in $H^{s}(\R^{N}, \R)$ (otherwise, by the continuity of $I_{0}$, $I_{0}(\tilde{v}_{n})\ri 0$ which is impossible because $d_{0}>0$). 
Then, from the uniqueness of the weak limit, we have that $\tilde{v}_{n}\rightharpoonup \tilde{v}=t^{*}v\not\equiv 0$ in $H^{s}(\R^{N}, \R)$. 
This fact combined with Lemma \ref{lem3.3MAW} yields
\begin{equation}\label{elena}
\tilde{v}_{n}\rightarrow \tilde{v} \quad \mbox{ in } H^{s}(\R^{N}, \R).
\end{equation} 
Consequently, $v_{n}\rightarrow v$ in $H^{s}(\R^{N}, \R)$ as $n\rightarrow \infty$. Moreover,
\begin{align}\label{elena15/0(}
I_{0}(\tilde{v})=0 \mbox{ and } \langle I'_{0}(\tilde{v}), \tilde{v}\rangle=0.
\end{align}
Now, we put $y_{n}=\e_{n}\tilde{y}_{n}$ and we claim that $(y_{n})$ admits a subsequence, still denoted by $(y_{n})$, such that $y_{n}\rightarrow y_{0}$ for some $y_{0}\in M$. First, we prove that $(y_{n})$ is bounded in $\R^{N}$. Assume by contradiction that, up to a subsequence, $|y_{n}|\rightarrow \infty$ as $n\rightarrow \infty$. Take $R>0$ such that $\Lambda \subset B_{R}$. Since we may suppose that  $|y_{n}|>2R$ for $n$ large, we have that for any $z\in B_{R/\e_{n}}$ 
$$
|\e_{n}z+y_{n}|\geq |y_{n}|-|\e_{n}z|>R.
$$
Hence, using $(u_{n})\subset \N_{\e_{n}}$, $(V_{1})$, Lemma \ref{DI} and the change of variable $x\mapsto z+\tilde{y}_{n}$, we obtain that 
\begin{align}\label{pasq}
[v_{n}]^{2}+\int_{\R^{N}} V_{0} v_{n}^{2}\, dx &\leq \int_{\R^{N}} g(\e_{n} x+y_{n}, |v_{n}|^{2}) |v_{n}|^{2} \, dx \nonumber\\
&\leq \int_{B_{\frac{R}{\e_{n}}}} \tilde{f}(|v_{n}|^{2}) |v_{n}|^{2} \, dx+\int_{\R^{N}\setminus B_{\frac{R}{\e_{n}}}} f(|v_{n}|^{2}) |v_{n}|^{2} \, dx.
\end{align}
Recalling that $v_{n}\rightarrow v$ in $H^{s}(\R^{N}, \R)$ as $n\rightarrow \infty$ and that $\tilde{f}(t)\leq \frac{V_{0}}{k}$, we can see that (\ref{pasq})  and the dominated convergence theorem yield
$$
\min\left\{1, V_{0}\left(1-\frac{1}{k}\right) \right\} \left([v_{n}]^{2}+\int_{\R^{N}} |v_{n}|^{2}\, dx\right)=o_{n}(1),
$$
that is $v_{n}\rightarrow 0$ in $H^{s}(\R^{N}, \R)$, which gives a contradiction. Therefore, $(y_{n})$ is bounded and we may assume that $y_{n}\rightarrow y_{0}\in \R^{N}$. If $y_{0}\notin \overline{\Lambda}$, then we can argue as before to infer that $v_{n}\rightarrow 0$ in $H^{s}(\R^{N}, \R)$, which is impossible. Hence $y_{0}\in \overline{\Lambda}$. Let us note that if $V(y_{0})=V_{0}$, then we can infer that $y_{0}\notin \partial \Lambda$ in view of $(V_2)$. Therefore, it is enough to verify that $V(y_{0})=V_{0}$ to deduce that $y_{0}\in M$. Suppose by contradiction that $V(y_{0})>V_{0}$.
Then, using (\ref{elena}), Fatou's lemma, the invariance of  $\R^{N}$ by translations and Lemma \ref{DI}, we get 
\begin{align*}
d_{0}=I_{0}(\tilde{v})&<\frac{1}{2}[\tilde{v}]^{2}+\frac{1}{2}\int_{\R^{N}} V(y_{0})\tilde{v}^{2} \, dx-\frac{1}{2}\int_{\R^{N}} F(|\tilde{v}|^{2})\, dx\\
&\leq \liminf_{n\rightarrow \infty}\left[\frac{1}{2}[\tilde{v}_{n}]^{2}+\frac{1}{2}\int_{\R^{N}} V(\e_{n}x+y_{n}) |\tilde{v}_{n}|^{2} \, dx-\frac{1}{2}\int_{\R^{N}} F(|\tilde{v}_{n}|^{2})\, dx  \right] \\
&\leq \liminf_{n\rightarrow \infty}\left[\frac{t_{n}^{2}}{2}[|u_{n}|]^{2}+\frac{t_{n}^{2}}{2}\int_{\R^{N}} V(\e_{n}z) |u_{n}|^{2} \, dz-\frac{1}{2}\int_{\R^{N}} F(|t_{n} u_{n}|^{2})\, dz  \right] \\
&\leq \liminf_{n\rightarrow \infty} J_{\e_{n}}(t_{n} u_{n}) \leq \liminf_{n\rightarrow \infty} J_{\e_{n}}(u_{n})\leq d_{0}
\end{align*}
which is impossible. This ends the proof of the lemma.
\end{proof}

\noindent
Now, we consider the following subset of $\N_{\e}$ 
$$
\widetilde{\N}_{\e}=\left \{u\in \N_{\e}: J_{\e}(u)\leq d_{0}+h(\e)\right\},
$$
where $h(\e)=\sup_{y\in M}|J_{\e}(\Phi_{\e}(y))-d_{0}|\rightarrow 0$ as $\e \rightarrow 0$ as a consequence of Lemma \ref{lem3.4}. By the definition of $h(\e)$, we know that, for all $y\in M$ and $\e>0$, $\Phi_{\e}(y)\in \widetilde{\N}_{\e}$ and $\widetilde{\N}_{\e}\neq \emptyset$. 
We present below an interesting relation between $\widetilde{\N}_{\e}$ and the barycenter map $\beta_{\e}$.
\begin{lem}\label{lem3.5}
For any $\delta>0$, there holds that
$$
\lim_{\e \rightarrow 0} \sup_{u\in \widetilde{\mathcal{N}}_{\e}} {\rm dist}(\beta_{\e}(u), M_{\delta})=0.
$$
\end{lem}
\begin{proof}
Let $(\e_{n})\subset (0, \infty)$ such that $\e_{n}\ri 0$. Then there exists $(u_{n})\subset \widetilde{\mathcal{N}}_{\e_{n}}$ such that 
$$
\sup_{u\in \widetilde{\N}_{\e_{n}}} \inf_{y\in M_{\delta}}|\beta_{\e_{n}}(u)-y|=\inf_{y\in M_{\delta}}|\beta_{\e_{n}}(u_{n})-y|+o_{n}(1).
$$
Hence, it is enough to find a sequence $(y_{n})\subset M_{\delta}$ such that
\begin{equation}\label{3.13}
\lim_{n\rightarrow \infty} |\beta_{\e_{n}}(u_{n})-y_{n}|=0.
\end{equation}
By Lemma \ref{DI} we know that $I_{0}(t |u_{n}|)\leq J(tu_{n})$ for all $t\geq 0$. Therefore, recalling that $(u_{n})\subset  \widetilde{\N}_{\e_{n}}\subset  \N_{\e_{n}}$, we deduce that
$$
d_{0}\leq \max_{t\geq 0} I_{0}(t|u_{n}|)\leq \max_{t\geq 0} J_{\e_{n}}(t u_{n})=J_{\e_{n}}(u_{n})\leq d_{0}+h(\e_{n})
$$
which leads to $J_{\e_{n}}(u_{n})\rightarrow d_{0}$. By invoking Lemma \ref{prop3.3}, we can find $(\tilde{y}_{n})\subset \R^{N}$ such that $y_{n}=\e_{n}\tilde{y}_{n}\in M_{\delta}$ for $n$ sufficiently large. Consequently,
$$
\beta_{\e_{n}}(u_{n})=y_{n}+\frac{\displaystyle{\int_{\R^{N}}[\varUpsilon(\e_{n}z+y_{n})-y_{n}] |u_{n}(z+\tilde{y}_{n})|^{2} \, dz}}{\displaystyle{\int_{\R^{N}} |u_{n}(z+\tilde{y}_{n})|^{2} \, dz}}.
$$
Taking into account that $|u_{n}|(\cdot+\tilde{y}_{n})$ strongly converges in $\mathcal{H}_{0}$ and that $\e_{n}z+y_{n}\rightarrow y_{0}\in M_{\delta}$, we find $\beta_{\e_{n}}(u_{n})=y_{n}+o_{n}(1)$, that is (\ref{3.13}) is satisfied.
\end{proof}

\noindent
We end this section by proving a multiplicity result for \eqref{MPe}. Since $\mathbb{S}^{+}_{\e}$ is not a completed metric space, we cannot use directly an abstract result as in \cite{Ana, Adcds, AD}. Instead, we invoke the abstract category result in \cite{SW}.
\begin{thm}\label{multiple}
For any $\delta>0$ such that $M_{\delta}\subset \Lambda$, there exists $\tilde{\e}_{\delta}>0$ such that, for any $\e\in (0, \tilde{\e}_{\delta})$, problem \eqref{MPe} has at least $cat_{M_{\delta}}(M)$ nontrivial solutions.
\end{thm}
\begin{proof}
For any $\e>0$, we consider the map $\alpha_{\e} : M \rightarrow \mathbb{S}_{\e}^{+}$ defined as $\alpha_{\e}(y)= m_{\e}^{-1}(\Phi_{\e}(y))$. \\
Using Lemma \ref{lem3.4}, we see that
\begin{equation}\label{FJSMAW}
\lim_{\e \rightarrow 0} \psi_{\e}(\alpha_{\e}(y)) = \lim_{\e \rightarrow 0} J_{\e}(\Phi_{\e}(y))= d_{0} \quad \mbox{ uniformly in } y\in M. 
\end{equation}  
Set
$$
\widetilde{\mathcal{S}}^{+}_{\e}=\{ w\in \mathbb{S}_{\e}^{+} : \psi_{\e}(w) \leq d_{0} + h(\e)\}, 
$$
where $h(\e)=\sup_{y\in \Lambda}|\psi_{\e}(\alpha_{\e}(y))-d_{0}|$.
It follows from \eqref{FJSMAW} that $h(\e)\rightarrow 0$ as $\e\rightarrow 0$. Moreover, $\alpha_{\e}(y)\in \widetilde{\mathcal{S}}^{+}_{\e}$ for all $y\in M$ and this shows that $\widetilde{\mathcal{S}}^{+}_{\e}\neq \emptyset$ for all $\e>0$. \\
In light of Lemma \ref{lem3.4}, Lemma \ref{lem2.3MAW}-$(iii)$, Lemma \ref{lem3.5N} and Lemma \ref{lem3.5}, we can find $\bar{\e}= \bar{\e}_{\delta}>0$ such that the following diagram
\begin{equation*}
M\stackrel{\Phi_{\e}}{\rightarrow} \Phi_{\e}(M) \stackrel{m_{\e}^{-1}}{\rightarrow} \alpha_{\e}(M)\stackrel{m_{\e}}{\rightarrow} \Phi_{\e}(M) \stackrel{\beta_{\e}}{\rightarrow} M_{\delta}
\end{equation*}    
is well defined for any $\e \in (0, \bar{\e})$. 
Thanks to Lemma \ref{lem3.5N}, and decreasing $\bar{\e}$ if necessary, we see that $\beta_{\e}(\Phi_{\e}(y))= y+ \theta(\e, y)$ for all $y\in M$, for some function $\theta(\e, y)$ satisfying $|\theta(\e, y)|<\frac{\delta}{2}$ uniformly in $y\in M$ and for all $\e \in (0, \bar{\e})$. Define $H(t, y)= y+ (1-t)\theta(\e, y)$. Then $H: [0,1]\times M\rightarrow M_{\delta}$ is continuous. Clearly,  $H(0, y)=\beta_{\e}(\Phi_{\e}(y))$ and $H(1,y)=y$ for all $y\in M$. Consequently, $H(t, y)$ is a homotopy between $\beta_{\e} \circ \Phi_{\e} = (\beta_{\e} \circ m_{\e}) \circ (m_{\e}^{-1}\circ \Phi_{\e})$ and the inclusion map $id: M \rightarrow M_{\delta}$. This fact yields 
\begin{equation}\label{catMAW}
cat_{\alpha_{\e}(M)} \alpha_{\e}(M)\geq cat_{M_{\delta}}(M).
\end{equation}
Applying Corollary \ref{cor2.1MAW}, Lemma \ref{AMlem1}, and Theorem $27$ in \cite{SW} with $c= c_{\e}\leq d_{0}+h(\e) =d$ and $K= \alpha_{\e}(M)$, we obtain that $\psi_{\e}$ has at least $cat_{\alpha_{\e}(M)} \alpha_{\e}(M)$ critical points on $\widetilde{\mathcal{S}}^{+}_{\e}$.
Taking into account Proposition \ref{prop2.1MAW}-$(d)$ and \eqref{catMAW}, we infer that $J_{\e}$ admits at least $cat_{M_{\delta}}(M)$ critical points in $\widetilde{\mathcal{N}}_{\e}$.    
\end{proof}

\section{Proof of Theorem \ref{thm1}}
This last section is devoted to the proof of the main result of this paper. In order to show that the solutions of \eqref{MPe} are indeed solutions to \eqref{R} for $\e>0$ small, we need to verify that $|u_{\e}|\leq \sqrt{a}$ in $\Lambda^{c}_{\e}$ holds true provided that $\e>0$ is sufficiently small.
We start by proving a fractional Kato's inequality in the spirit of \cite{Kato} for the solutions of fractional magnetic problems.
\begin{thm}\label{KATOTHMNG}
Let $u\in \mathcal{D}^{s, 2}_{A}(\R^{N}, \C)$ and $f\in L^{1}_{loc}(\R^{N}, \C)$ be such that
\begin{align}\label{wffNG}
&\Re\left(\iint_{\R^{2N}} \frac{(u(x)-u(y)e^{\imath A(\frac{x+y}{2})\cdot (x-y)})}{|x-y|^{N+2s}} \overline{(\psi(x)-\psi(y)e^{\imath A(\frac{x+y}{2})\cdot (x-y)})}  \,dxdy \right) \nonumber\\
&\quad=\Re\left(\int_{\R^{N}} f\bar{\psi} \,dx \right)
\end{align}
for all $\psi:\R^{N}\rightarrow \C$ measurable with compact support and such that $[\psi]_{A}<\infty$.

Then it holds
$(-\Delta)^{s}|u|\leq \Re({\rm sign}(\bar{u})f)$ in the distributional sense, that is
\begin{align}\label{Kato0NG}
\iint_{\R^{2N}} \frac{(|u(x)|-|u(y)|)(\varphi(x)-\varphi(y))}{|x-y|^{N+2s}} \,dx dy\leq \Re\left(\int_{\R^{N}} {\rm sign}(\bar{u})f \varphi \,dx \right)
\end{align}
for all $\varphi\in C^{\infty}_{c}(\R^{N}, \R)$ such that $\varphi\geq 0$, where 
\begin{equation*}
{\rm sign}(\bar{u})(x)=
\left\{ 
\begin{array}{cc}
\frac{\overline{u(x)}}{|u(x)|} &\mbox{ if } u(x)\neq 0, \\
0 &\mbox{ if } u(x)=0.
\end{array}
\right.
\end{equation*}
\end{thm}
\begin{proof}
Take $\varphi\in C^{\infty}_{c}(\R^{N}, \R)$ such that $\varphi\geq 0$, and for all $\delta>0$ we consider 
$$
\omega_{\delta}(x)=\frac{u(x)}{\sqrt{|u(x)|^{2}+\delta^{2}}} \varphi(x)= \frac{u(x)}{u_{\delta}(x)}\varphi(x)
$$ 
as test function in \eqref{wffNG}. First, we show that $\omega_{\delta}$ is admissible.
It is clear that $\omega_{\delta}$ has compact support.    
On the other hand, we can observe
\begin{align*}
\omega_{\delta}(x)-\omega_{\delta}(y)e^{\imath A(\frac{x+y}{2})\cdot (x-y)}&=\left(\frac{u(x)}{u_{\delta}(x)}\right)\varphi(x)-\left(\frac{u(y)}{u_{\delta}(y)}\right)\varphi(y)e^{\imath A(\frac{x+y}{2})\cdot (x-y)}\\
&=\left[u(x)-u(y)e^{\imath A(\frac{x+y}{2})\cdot (x-y)}\right] \frac{\varphi(x)}{u_{\delta}(x)}\\
&+\left[\frac{\varphi(x)}{u_{\delta}(x)}-\frac{\varphi(y)}{u_{\delta}(y)} \right] u(y)e^{\imath A(\frac{x+y}{2})\cdot (x-y)} \\
&=\left[u(x)-u(y)e^{\imath A(\frac{x+y}{2})\cdot (x-y)}\right] \frac{\varphi(x)}{u_{\delta}(x)}\\
&+\left[\frac{1}{u_{\delta}(x)}-\frac{1}{u_{\delta}(y)} \right] \varphi(x)u(y)e^{\imath A(\frac{x+y}{2})\cdot (x-y)}\\
&+[\varphi(x)-\varphi(y)] \frac{u(y)}{u_{\delta}(y)} e^{\imath A(\frac{x+y}{2})\cdot (x-y)}
\end{align*}
which implies that
\begin{align*}
&|\omega_{\delta}(x)-\omega_{\delta}(y)e^{\imath A(\frac{x+y}{2})\cdot (x-y)}|^{2} \\
&\leq \frac{4}{\delta^{2}}|u(x)-u(y)e^{\imath A(\frac{x+y}{2})\cdot (x-y)}|^{2}\|\varphi\|^{2}_{L^{\infty}(\R^{N})}\\
&+4\left|\frac{u(y)}{u_{\delta}(y)}\right|^{2} \frac{1}{|u_{\delta}(x)|^{2}}  \|\varphi\|^{2}_{L^{\infty}(\R^{N})} |u_{\delta}(y)-u_{\delta}(x)|^{2} +4|\varphi(x)-\varphi(y)|^{2} \\
&\leq \frac{4}{\delta^{2}}|u(x)-u(y)e^{\imath A(\frac{x+y}{2})\cdot (x-y)}|^{2}\|\varphi\|^{2}_{L^{\infty}(\R^{N})}+\frac{4}{\delta^{2}}||u(x)|-|u(y)||^{2}\|\varphi\|^{2}_{L^{\infty}(\R^{N})}\\
&+4|\varphi(x)-\varphi(y)|^{2},
\end{align*}
where we used the following elementary inequalities
\begin{align*}
&|z+w+k|^{2}\leq 4(|z|^{2}+|w|^{2}+|k|^{2}) \quad   \mbox{ for all } z,w,k\in \C, \\
&|\sqrt{|z|^{2}+\delta^{2}}-\sqrt{|w|^{2}+\delta^{2}}|\leq ||z|-|w||  \quad \mbox{ for all } z, w\in \C,
\end{align*}
and that $|e^{\imath t}|=1$ for all $t\in \R$, $u_{\delta}\geq \delta$, $|\frac{u}{u_{\delta}}|\leq 1$. 
Since $u\in \mathcal{D}^{s,2}_{A}(\R^{N}, \C)$, $|u|\in \mathcal{D}^{s,2}(\R^{N}, \R)$ (by Lemma \ref{DI}) and $\varphi\in C^{\infty}_{c}(\R^{N}, \R)$, we deduce that $[\omega_{\delta}]_{A}<\infty$.
Then we have
\begin{align}\label{Kato1NG}
&\Re\left[\iint_{\R^{2N}} \frac{(u(x)-u(y)e^{\imath A(\frac{x+y}{2})\cdot (x-y)})}{|x-y|^{N+2s}} \left(\frac{\overline{u(x)}}{u_{\delta}(x)}\varphi(x)-\frac{\overline{u(y)}}{u_{\delta}(y)}\varphi(y)e^{-\imath A(\frac{x+y}{2})\cdot (x-y)}  \right) dx dy\right] \nonumber\\
&=\Re\left(\int_{\R^{N}} f\frac{\overline{u}}{u_{\delta}}\varphi \,dx\right).
\end{align}
Now, using $\Re(z)\leq |z|$ for all $z\in \C$ and  $|e^{\imath t}|=1$ for all $t\in \R$, we see that
\begin{align}\label{alves1NG}
&\Re\left[(u(x)-u(y)e^{\imath A(\frac{x+y}{2})\cdot (x-y)}) \left(\frac{\overline{u(x)}}{u_{\delta}(x)}\varphi(x)-\frac{\overline{u(y)}}{u_{\delta}(y)}\varphi(y)e^{-\imath A(\frac{x+y}{2})\cdot (x-y)}  \right)\right] \nonumber\\
&=\Re\left[\frac{|u(x)|^{2}}{u_{\delta}(x)}\varphi(x)+\frac{|u(y)|^{2}}{u_{\delta}(y)}\varphi(y)-\frac{u(x)\overline{u(y)}}{u_{\delta}(y)}\varphi(y)e^{-\imath A(\frac{x+y}{2})\cdot (x-y)} -\frac{u(y)\overline{u(x)}}{u_{\delta}(x)}\varphi(x)e^{\imath A(\frac{x+y}{2})\cdot (x-y)}\right] \nonumber \\
&\geq \left[\frac{|u(x)|^{2}}{u_{\delta}(x)}\varphi(x)+\frac{|u(y)|^{2}}{u_{\delta}(y)}\varphi(y)-|u(x)|\frac{|u(y)|}{u_{\delta}(y)}\varphi(y)-|u(y)|\frac{|u(x)|}{u_{\delta}(x)}\varphi(x) \right].
\end{align}
Let us note that
\begin{align}\label{alves2NG}
&\frac{|u(x)|^{2}}{u_{\delta}(x)}\varphi(x)+\frac{|u(y)|^{2}}{u_{\delta}(y)}\varphi(y)-|u(x)|\frac{|u(y)|}{u_{\delta}(y)}\varphi(y)-|u(y)|\frac{|u(x)|}{u_{\delta}(x)}\varphi(x) \nonumber\\
&=  \frac{|u(x)|}{u_{\delta}(x)}(|u(x)|-|u(y)|)\varphi(x)-\frac{|u(y)|}{u_{\delta}(y)}(|u(x)|-|u(y)|)\varphi(y) \nonumber\\
&=\left[\frac{|u(x)|}{u_{\delta}(x)}(|u(x)|-|u(y)|)\varphi(x)-\frac{|u(x)|}{u_{\delta}(x)}(|u(x)|-|u(y)|)\varphi(y)\right] \nonumber\\
&+\left(\frac{|u(x)|}{u_{\delta}(x)}-\frac{|u(y)|}{u_{\delta}(y)} \right) (|u(x)|-|u(y)|)\varphi(y) \nonumber\\
&=\frac{|u(x)|}{u_{\delta}(x)}(|u(x)|-|u(y)|)(\varphi(x)-\varphi(y)) +\left(\frac{|u(x)|}{u_{\delta}(x)}-\frac{|u(y)|}{u_{\delta}(y)} \right) (|u(x)|-|u(y)|)\varphi(y) \nonumber\\
&\geq \frac{|u(x)|}{u_{\delta}(x)}(|u(x)|-|u(y)|)(\varphi(x)-\varphi(y)), 
\end{align}
where in the last inequality we used the fact that
$$
\left(\frac{|u(x)|}{u_{\delta}(x)}-\frac{|u(y)|}{u_{\delta}(y)} \right) (|u(x)|-|u(y)|)\varphi(y)\geq 0
$$
because
$$
h(t)=\frac{t}{\sqrt{t^{2}+\delta^{2}}} \mbox{ is increasing for } t\geq 0 \quad \mbox{ and } \quad \varphi\geq 0 \mbox{ in }\R^{N}.
$$
Since
$$
\frac{|\frac{|u(x)|}{u_{\delta}(x)}(|u(x)|-|u(y)|)(\varphi(x)-\varphi(y))|}{|x-y|^{N+2s}}\leq \frac{||u(x)|-|u(y)||}{|x-y|^{\frac{N+2s}{2}}} \frac{|\varphi(x)-\varphi(y)|}{|x-y|^{\frac{N+2s}{2}}}\in L^{1}(\R^{2N}, \R),
$$
and $\frac{|u(x)|}{u_{\delta}(x)}\rightarrow 1$ a.e. in $\R^{N}$ as $\delta\rightarrow 0$,
we can use \eqref{alves1NG}, \eqref{alves2NG} and the dominated convergence theorem to deduce that
\begin{align}\label{Kato2NG}
&\liminf_{\delta\rightarrow 0} \Re\left[\iint_{\R^{2N}} \frac{(u(x)-u(y)e^{\imath A(\frac{x+y}{2})\cdot (x-y)})}{|x-y|^{N+2s}} \left(\frac{\overline{u(x)}}{u_{\delta}(x)}\varphi(x)-\frac{\overline{u_{n}(y)}}{u_{\delta}(y)}\varphi(y)e^{-\imath A(\frac{x+y}{2})\cdot (x-y)}  \right) \,dx dy\right] \nonumber\\
&\geq \liminf_{\delta\rightarrow 0} \iint_{\R^{2N}} \frac{|u(x)|}{u_{\delta}(x)}\, \frac{(|u(x)|-|u(y)|)(\varphi(x)-\varphi(y))}{|x-y|^{N+2s}} \ dx dy  \nonumber\\
&=\iint_{\R^{2N}} \frac{(|u(x)|-|u(y)|)(\varphi(x)-\varphi(y))}{|x-y|^{N+2s}} \,dx dy.
\end{align}
On the other hand, observing that $|f\frac{\overline{u}}{u_{\delta}}\varphi|\leq |f\varphi|\in L^{1}(\R^{N}, \R)$ and $ f\frac{\overline{u}}{u_{\delta}}\varphi\rightarrow  f {\rm sign}(\bar{u})\varphi$ a.e. in $\R^{N}$ as $\delta\rightarrow 0$, we can invoke the dominated convergence theorem to infer that as $\delta\rightarrow 0$
\begin{align}\label{Kato3NG}
\Re\left(\int_{\R^{N}} f\frac{\overline{u}}{u_{\delta}}\varphi \,dx\right)\rightarrow \Re\left(\int_{\R^{N}} f {\rm sign}(\bar{u}) \varphi \,dx\right).
\end{align}
Putting together \eqref{Kato1NG}, \eqref{Kato2NG} and \eqref{Kato3NG}, we see that \eqref{Kato0NG} holds true.
\end{proof}

\begin{remark}
A pointwise Kato's inequality  for $(-\Delta)^{s}_{A}$ is proved in \cite{Acpde}. In \cite{HIL} the authors established a Kato's inequality for the fractional magnetic operator $((-\imath\nabla-A(x))^{2}+m^{2})^{\frac{\alpha}{2}}$ with $\alpha\in (0, 1]$ and $m> 0$, or $\alpha=1$ and $m=0$, borrowing some arguments used in \cite{Kato}. As observed in \cite{DS}, when $\alpha= 1$ and $m=0$, this operator coincides with $(-\Delta)^{\frac{1}{2}}_{A}$.
\end{remark}

Now we prove the following crucial result.
\begin{lem}\label{moser} 
Let $\e_{n}\rightarrow 0$ and $(u_{n})\subset \widetilde{\mathcal{N}}_{\e_{n}}$ be a sequence of solutions to \eqref{MPe}. 
Then, $v_{n}=|u_{n}|(\cdot+\tilde{y}_{n})$ satisfies $v_{n}\in L^{\infty}(\R^{N},\R)$ and there exists $C>0$ such that 
$$
\|v_{n}\|_{L^{\infty}(\R^{N})}\leq C \quad \mbox{ for all } n\in \mathbb{N},
$$
where $(\tilde{y}_{n})\subset \R^{N}$ is given by Lemma \ref{prop3.3}.
Moreover,
$$
\lim_{|x|\rightarrow \infty} v_{n}(x)=0 \quad \mbox{ uniformly in } n\in \mathbb{N}.
$$
\end{lem}
\begin{proof}
Since $J_{\e_{n}}(u_{n})\leq d_{0}+h(\e_{n})$ with $h(\e_{n})\rightarrow 0$ as $n\rightarrow \infty$, we can argue as at the beginning of the proof of Lemma \ref{prop3.3} to deduce that $J_{\e_{n}}(u_{n})\rightarrow d_{0}$. Thus we may invoke  Lemma \ref{prop3.3} to obtain a sequence $(\tilde{y}_{n})\subset \R^{N}$ such that $\e_{n}\tilde{y}_{n}\rightarrow y_{0}\in M$ and $v_{n}=|u_{n}|(\cdot+\tilde{y}_{n})$ strongly converges in $\mathcal{H}_{0}$. Let $\tilde{u}_{n}(x)=u_{n}(\cdot+\tilde{y}_{n})$ and note that it solves 
\begin{equation}\label{tetrapak}
(-\Delta)^{s}_{\tilde{A}_{n}} \tilde{u}_{n}+\tilde{V}_{n}(x) \tilde{u}_{n}=\tilde{g}_{n}(x, v^{2}_{n}) \tilde{u}_{n} \quad \mbox{ in } \R^{N},
\end{equation}
where
$$
\tilde{A}_{n}(x)=A_{\e_{n}}(x+\tilde{y}_{n}),
$$
$$
\tilde{V}_{n}(x)=V_{\e_{n}}(x+\tilde{y}_{n}),
$$
and
$$
\tilde{g}_{n}(x, v^{2}_{n})=g(\e_{n} x+\e_{n}\tilde{y}_{n}, v_{n}^{2}(x)).
$$
Using Theorem \ref{KATOTHMNG}, we deduce that $v_{n}$ satisfies (in the distributional sense)
\begin{equation}\label{SUB}
(-\Delta)^{s} v_{n}+\tilde{V}_{n}(x) v_{n}\leq \tilde{g}_{n}(x, v^{2}_{n}) v_{n}=h_{n}  \quad \mbox{ in }  \R^{N}.
\end{equation}
By performing a Moser iteration argument \cite{Moser} as in Lemma $5.1$ in \cite{A3}, we obtain that  
\begin{equation}\label{UBu}
\|v_{n}\|_{L^{\infty}(\mathbb{R}^{N})}\leq K \quad \mbox{ for all } n\in \mathbb{N}.
\end{equation}
Moreover, by interpolation, $v_{n}\rightarrow v$ strongly converges  in $L^{r}(\mathbb{R}^{N}, \mathbb{R})$ for all $r\in [2, \infty)$. In view of the growth assumptions on $g$, we can also see that $h_{n}\rightarrow  h=f(v^{2})v$ in $L^{r}(\mathbb{R}^{N}, \mathbb{R})$ for all $r\in [2, \infty)$, and $\|h_{n}\|_{L^{\infty}(\mathbb{R}^{N})}\leq C$ for all $n\in \mathbb{N}$.
Note that, by $(V_{1})$, we have
\begin{equation}\label{Pkat}
(-\Delta)^{s} v_{n} + V_{0}v_{n}\leq h_{n} \quad \mbox{ in } \mathbb{R}^{N}. 
\end{equation}
Let us denote by $z_{n}\in H^{s}(\mathbb{R}^{N}, \mathbb{R})$ the unique solution to
\begin{equation}\label{US}
(-\Delta)^{s} z_{n} + V_{0}z_{n}=h_{n} \quad \mbox{ in } \mathbb{R}^{N}.
\end{equation}
Since $v_{n}$ satisfies \eqref{Pkat} and $z_{n}$ solves \eqref{US}, by comparison we see that $0\leq v_{n}\leq z_{n}$ in $\mathbb{R}^{N}$ for all $n\in \mathbb{N}$. 
Next we show that
\begin{align}\label{ZCONVERGENZA}
z_{n}(x)\ri 0 \,\mbox{ as } |x|\ri \infty \,\mbox{ uniformly in } n\in \mathbb{N}.
\end{align}
Note that $z_{n}(x)=(\mathcal{K}*h_{n})(x)$, where the kernel $\mathcal{K}(x)=\mathcal{F}^{-1}((|k|^{2s}+V_{0})^{-1})$
 satisfies the following properties (see \cite{FQT}):
\begin{compactenum}[$(b_1)$]
\item $\mathcal{K}$ is positive, radially symmetric and smooth in $\R^{N}\setminus \{0\}$;
\item there exists a positive constant $C>0$ such that $\mathcal{K}(x)\leq \frac{K_{1}}{|x|^{N+2s}}$ for any $x\in \R^{N}\setminus \{0\}$;
\item $\mathcal{K}\in L^{\nu}(\R^{N})$ for any $\nu\in [1, \frac{N}{N-2s})$.
\end{compactenum}
Now we borrow some arguments used in \cite{AM} to prove that \eqref{ZCONVERGENZA} holds true.
Fix $\delta>0$ and we observe that
\begin{align}\label{ASPRITE}
0\leq z_{n}(x)=(\mathcal{K}*h_{n})(x)=\int_{B^{c}_{\frac{1}{\delta}}(x)} \mathcal{K}(x-y)h_{n}(y) \, dy+\int_{B_{\frac{1}{\delta}}(x)} \mathcal{K}(x-y)h_{n}(y) \, dy.
\end{align}
From $(b_2)$ we deduce that 
\begin{align}\label{A1SPRITE}
\int_{B^{c}_{\frac{1}{\delta}}(x)} \mathcal{K}(x-y)h_{n}(y) \, dy \leq K_{1}\|h_{n}\|_{L^{\infty}(\R^{N})}\int_{B^{c}_{\frac{1}{\delta}}(x)} \frac{dy}{|x-y|^{N+2s}}\leq c_{1}\delta^{2s} K_{1} \int_{|\xi|\geq 1} \frac{d\xi}{|\xi|^{N+2s}}= C_{1}\delta^{2s}.
\end{align}
On the other hand, 
\begin{align*}
\left|\int_{B_{\frac{1}{\delta}}(x)} \mathcal{K}(x-y)h_{n}(y) \, dy\right|\leq \int_{B_{\frac{1}{\delta}}(x)} \mathcal{K}(x-y)|h_{n}(y)-h(y)| \, dy+
\int_{B_{\frac{1}{\delta}}(x)} \mathcal{K}(x-y)|h(y)| \, dy.
\end{align*}
Fix
$$
\nu\in \left(1, \min\left\{\frac{N}{N-2s}, 2 \right\}\right).
$$ 
Note that, if $N\geq 4s$ then $\nu\in (1, \frac{N}{N-2s})$ and we have $\nu'=\frac{\nu}{\nu-1}>\frac{N}{2s}\geq 2$, and that when $2s<N<4s$ then $2<\frac{N}{N-2s}$, $\nu\in (1, 2)$ and $\nu'\in (2, \infty)$. 
In any case, $\nu\in (1, \frac{N}{N-2s})$ and $\nu'\in (2, \infty)$.
Then, by $(b_{3})$, we deduce that $\mathcal{K}\in L^{\nu}(\R^{N})$. Using H\"older's inequality we have
\begin{align*}
\left|\int_{B_{\frac{1}{\delta}}(x)} \mathcal{K}(x-y)h_{n}(y) \, dy\right|\leq \|\mathcal{K}\|_{L^{\nu}(\R^{N})} \|h_{n}-h\|_{L^{\nu'}(\R^{N})}+\|\mathcal{K}\|_{L^{\nu}(\R^{N})} \|h\|_{L^{\nu'}(B_{\frac{1}{\delta}}(x))}.
\end{align*}
Since $\|h_{n}-h\|_{L^{\nu'}(\R^{N})}\rightarrow 0$ as $n\rightarrow \infty$ and $\|h\|_{L^{\nu'}(B_{\frac{1}{\delta}}(x))}\rightarrow 0$ as $|x|\rightarrow \infty$, we see that  there exist $R>0$ and $n_{0}\in \mathbb{N}$ such that
\begin{align}\label{A2SPRITE}
\int_{B_{\frac{1}{\delta}}(x)} \mathcal{K}(x-y)h_{n}(y) \, dy\leq C_{2}\delta \quad \mbox{ for all } |x|\geq R, \, n\geq n_{0}.
\end{align}
Putting together \eqref{ASPRITE}, \eqref{A1SPRITE} and \eqref{A2SPRITE}, we obtain that
\begin{align}\label{A3SPRITE}
z_{n}(x)=\int_{\R^{N}} \mathcal{K}(x-y)h_{n}(y) \, dy\leq C_{1}\delta^{2s}+C_{2}\delta \quad \mbox{ for all } |x|\geq R, \, n\geq n_{0}.
\end{align}
Now, for each $n\in \{1, \dots, n_{0}-1\}$, there is $R_{n}>0$ such that $\|h_{n}\|_{L^{\nu'}(B_{\frac{1}{\delta}}(x))}<\delta$ for $|x|\geq R_{n}$. Then, for $|x|\geq R_{n}$, we have
\begin{align}\label{A4SPRITE}
\int_{\R^{N}} \mathcal{K}(x-y)h_{n}(y) \, dy&\leq C_{1} \delta^{2s}+\int_{B_{\frac{1}{\delta}}(x)} \mathcal{K}(x-y)h_{n}(y) \, dy\nonumber \\
&\leq C_{1} \delta^{2s}+\|\mathcal{K}\|_{L^{\nu}(\R^{N})} \|h_{n}\|_{L^{\nu'}(B_{\frac{1}{\delta}}(x))} \nonumber \\
&\leq C_{1}\delta^{2s}+C_{3}\delta.
\end{align}
Set $\bar{R}=\max\{R_{1}, \dots, R_{n_{0}-1}, R\}$. By \eqref{A3SPRITE} and \eqref{A4SPRITE} we find
\begin{align*}
z_{n}(x)\leq C_{4}(\delta^{2s}+\delta) \quad \mbox{ for all } |x|\geq \bar{R}, \, n\in \mathbb{N}.
\end{align*}
From the arbitrariness of $\delta>0$, we deduce that \eqref{ZCONVERGENZA} holds true. 
Consequently, combining \eqref{ZCONVERGENZA} with $0\leq v_{n}\leq z_{n}$ in $\R^{N}$, we obtain that 
 $v_{n}(x)\rightarrow 0$ as $|x|\rightarrow \infty$ uniformly with respect to $n\in \mathbb{N}$.
\end{proof}

\noindent
Now, we are ready to give the proof of Theorem \ref{thm1}.
\begin{proof}[Proof of Theorem \ref{thm1}]
Let $\delta>0$ be such that $M_{\delta}\subset \Lambda$. First, we claim that there exists  $\hat{\e}_{\delta}>0$ such that for any $\e\in (0, \hat{\e}_{\delta})$ and any solution $u\in \widetilde{\mathcal{N}}_{\e}$ of \eqref{MPe}, it holds
\begin{equation}\label{Ua}
\|u\|_{L^{\infty}(\Lambda^{c}_{\e})}<\sqrt{a}.
\end{equation}
We argue by contradiction and assume that for some subsequence $\e_{n}\rightarrow 0$ we can obtain $u_{n}=u_{\e_{n}}\in \widetilde{\mathcal{N}}_{\e_{n}}$ such that $J'_{\e}(u_{n})=0$ and
\begin{equation}\label{AbsAFF}
\|u_{n}\|_{L^{\infty}(\Lambda^{c}_{\e})}\geq \sqrt{a}.
\end{equation}
Since $J_{\e_{n}}(u_{n})\leq d_{0}+h(\e_{n})$, we can argue as in the first part of Lemma \ref{prop3.3} to deduce that $J_{\e_{n}}(u_{n})\rightarrow d_{0}$.
In view of Lemma \ref{prop3.3}, there exists $(\tilde{y}_{n})\subset \R^{N}$ such that $\e_{n}\tilde{y}_{n}\rightarrow y_{0}$ for some $y_{0} \in M$ and $v_{n}=|u_{n}|(\cdot+\tilde{y}_{n})$ strongly converges in $H^{s}(\R^{N}, \R)$. 


Take $r>0$ such that, for some subsequence still denoted by itself, it holds $B_{r}(\e_{n}\tilde{y}_{n})\subset \Lambda$ for all $n\in \mathbb{N}$.
Hence $B_{\frac{r}{\e_{n}}}(\tilde{y}_{n})\subset \Lambda_{\e_{n}}$ for all  $n\in \mathbb{N}$, and consequently 
$\Lambda^{c}_{\e_{n}}\subset  B^{c}_{\frac{r}{\e_{n}}}(\tilde{y}_{n})$ for all $n\in \mathbb{N}$. 
By Lemma \ref{moser}, we can find $R>0$ such that 
$$
v_{n}(x)<\sqrt{a} \quad \mbox{ for all } |x|\geq R, \,n\in \mathbb{N},
$$ 
from which we deduce that $|u_{n}(x)|<\sqrt{a}$ for any $x\in B^{c}_{R}(\tilde{y}_{n})$ and $n\in \mathbb{N}$. On the other hand, there exists $\nu \in \mathbb{N}$ such that for any $n\geq \nu$ it holds 
$$
\Lambda^{c}_{\e_{n}}\subset B^{c}_{\frac{r}{\e_{n}}}(\tilde{y}_{n})\subset B^{c}_{R}(\tilde{y}_{n}).
$$ 
Therefore, $|u_{n}(x)|<\sqrt{a}$ for any $x\in\Lambda^{c}_{\e_{n}}$ and $n\geq \nu$, and this contradicts \eqref{AbsAFF}.

Let $\tilde{\e}_{\delta}>0$ be given by Theorem \ref{multiple} and we set $\e_{\delta}=\min\{\tilde{\e}_{\delta}, \hat{\e}_{\delta} \}$. Let us fix $\e\in (0, \e_{\delta})$. Applying Theorem \ref{multiple} we obtain at least $cat_{M_{\delta}}(M)$ nontrivial solutions to \eqref{MPe}.
If $u_{\e}\in \h$ is one of these solutions, then $u_{\e}\in \widetilde{\mathcal{N}}_{\e}$, and in view of \eqref{Ua} and the definition of $g$, we  infer that $u_{\e}$ is also a solution to \eqref{R}. Since $\hat{u}_{\e}(x)=u_{\e}(x/\e)$ is a solution to (\ref{P}), we deduce that \eqref{P} has at least $cat_{M_{\delta}}(M)$ nontrivial solutions.
Finally, we investigate the behavior of the maximum points of  $|\hat{u}_{\e}|$. Take $\e_{n}\rightarrow 0$ and let $(u_{n})\subset \mathcal{H}_{\e_{n}}$ be a sequence of solutions to \eqref{MPe} as above. From $(g_1)$, there exists $\gamma\in (0, \sqrt{a})$ such that
\begin{align}\label{4.18HZ}
g_{\e} (x, t^{2})t^{2}\leq \frac{V_{0}}{2}t^{2} \quad \mbox{ for all } x\in \R^{N}, |t|\leq \gamma.
\end{align}
Arguing as above, we can find $R>0$ such that
\begin{align}\label{4.19HZ}
\|u_{n}\|_{L^{\infty}(B^{c}_{R}(\tilde{y}_{n}))}<\gamma.
\end{align}
Up to a subsequence, we may also assume that
\begin{align}\label{4.20HZ}
\|u_{n}\|_{L^{\infty}(B_{R}(\tilde{y}_{n}))}\geq \gamma.
\end{align}
Indeed, if \eqref{4.20HZ} does not hold, we get $\|u_{n}\|_{L^{\infty}(\R^{N})}< \gamma$, and using $J_{\e_{n}}'(u_{n})=0$, \eqref{4.18HZ} and Lemma \ref{DI}, we deduce that 
\begin{align*}
[|u_{n}|]^{2}+\int_{\R^{N}}V_{0}|u_{n}|^{2}dx&\leq \|u_{n}\|^{2}_{\e_{n}}=\int_{\R^{N}} g_{\e_{n}}(x, |u_{n}|^{2})|u_{n}|^{2}\,dx\leq \frac{V_{0}}{2}\int_{\R^{N}}|u_{n}|^{2}\, dx.
\end{align*}
This fact yields $\||u_{n}|\|_{0}=0$, which is impossible. Hence, \eqref{4.20HZ} is satisfied.

In light of \eqref{4.19HZ} and \eqref{4.20HZ}, we can see that if $p_{n}$ is a global maximum point of $|u_{n}|$, then  $p_{n}$ belongs to $B_{R}(\tilde{y}_{n})$, that is $p_{n}=\tilde{y}_{n}+q_{n}$ for some $q_{n}\in B_{R}$. Since $\hat{u}_{n}(x)=u_{n}(x/\e_{n})$ is solution of \eqref{P}, we infer that $\eta_{n}=\e_{n}\tilde{y}_{n}+\e_{n}q_{n}$ is a global maximum point of $|\hat{u}_{n}|$. Using $(q_{n})\subset B_{R}$, $\e_{n}\tilde{y}_{n}\rightarrow y_{0}\in M$, and the continuity of $V$, we deduce that
$$
\lim_{n\rightarrow \infty} V(\eta_{n})=V(y_{0})=V_{0}.
$$
Finally, we provide a decay estimate for $|\hat{u}_{n}|$.
By using Lemma $4.3$ in \cite{FQT} there exists a continuous function $w$ such that
\begin{align}\label{HZ1}
0<w(x)\leq \frac{C}{1+|x|^{N+2s}} \quad \mbox{ for all } x\in \R^{N},
\end{align}
and satisfying in the classical sense
\begin{align}\label{HZ2}
(-\Delta)^{s} w+\frac{V_{0}}{2}w= 0  \quad \mbox{ in } \mathbb{R}^{N}\setminus \overline{B}_{R_{1}}, 
\end{align}
for some suitable $R_{1}>0$. By Lemma \ref{moser}, we can find $R_{2}>0$ such that
\begin{equation}\label{hzero}
h_{n}=g(\e_{n}\cdot+\e_{n}\tilde{y}_{n}, v_{n}^{2})v_{n}\leq \frac{V_{0}}{2}v_{n}  \quad \mbox{ in } \mathbb{R}^{N}\setminus \overline{B}_{R_{2}}.
\end{equation}
Let us denote by $w_{n}$ the unique solution to 
$$
(-\Delta)^{s}w_{n}+V_{0}w_{n}=h_{n}  \quad \mbox{ in } \mathbb{R}^{N}.
$$
Then, arguing as in the proof of \eqref{ZCONVERGENZA}, we see that $w_{n}(x)\rightarrow 0$ as $|x|\rightarrow \infty$ uniformly in $n\in \mathbb{N}$. Thus, by comparison, we obtain that $0\leq v_{n}\leq w_{n}$ in $\mathbb{R}^{N}$. Moreover, in light of \eqref{hzero}, it holds
\begin{align*}
(-\Delta)^{s}w_{n}+\frac{V_{0}}{2}w_{n}=h_{n}-\frac{V_{0}}{2}w_{n}\leq 0 \quad \mbox{ in } \mathbb{R}^{N}\setminus \overline{B}_{R_{2}}.
\end{align*}
Take $R_{3}=\max\{R_{1}, R_{2}\}$ and we set 
\begin{align}\label{HZ4}
c=\min_{\overline{B}_{R_{3}}} w>0 \quad  \mbox{ and } \quad \tilde{w}_{n}=(b+1)w-c w_{n},
\end{align}
where $b=\sup_{n\in \mathbb{N}} \|w_{n}\|_{L^{\infty}(\mathbb{R}^{N})}<\infty$. 
Our aim is to prove that 
\begin{equation}\label{HZ5}
\tilde{w}_{n}\geq 0 \quad \mbox{ in } \mathbb{R}^{N}.
\end{equation}
First, we observe that
\begin{align*}
&\tilde{w}_{n}\geq bc+w-bc>0  \quad\mbox{ in } \overline{B}_{R_{3}}, \\
&(-\Delta)^{s} \tilde{w}_{n}+\frac{V_{0}}{2}\tilde{w}_{n}\geq 0 \quad \mbox{ in } \mathbb{R}^{N}\setminus \overline{B}_{R_{3}}. 
\end{align*}
Then we can apply a comparison principle (see Theorem $7.1$ in \cite{DPQNA}) to deduce that \eqref{HZ5} holds true.
From (\ref{HZ1}), (\ref{HZ5}) and $0\leq v_{n}\leq w_{n}$ in $\R^{N}$, we get
\begin{align*}
0\leq v_{n}(x)\leq w_{n}(x)\leq \frac{(b+1)}{c}w(x)\leq \frac{\tilde{C}}{1+|x|^{N+2s}} \quad \mbox{ for all } x\in \mathbb{R}^{N}, n\in \mathbb{N}, 
\end{align*}
for some constant $\tilde{C}>0$. Therefore, recalling the definition of $v_{n}$, we infer that  
\begin{align*}
|\hat{u}_{n}|(x)&=|u_{n}|\left(\frac{x}{\e_{n}}\right)=v_{n}\left(\frac{x}{\e_{n}}-\tilde{y}_{n}\right) \\
&\leq \frac{\tilde{C}}{1+|\frac{x}{\e_{n}}-\tilde{y}_{n}|^{N+2s}} \\
&=\frac{\tilde{C} \e_{n}^{N+2s}}{\e_{n}^{N+2s}+|x- \e_{n} \tilde{y}_{n}|^{N+2s}} \\
&\leq \frac{\tilde{C} \e_{n}^{N+2s}}{\e_{n}^{N+2s}+|x-\eta_{n}|^{N+2s}} \quad \mbox{ for all } x\in \R^{N}.
\end{align*}
This ends the proof of Theorem \ref{thm1}.
\end{proof}


\begin{remark}
The approach used in this paper can be easily extended to deal with fractional Kirchhoff problems \cite{Adcds}, fractional Choquard equations \cite{Adypde} and fractional Schr\"odinger-Poisson systems \cite{Apisa}.
\end{remark}

\noindent

\end{document}